\documentclass[journal]{IEEEtran}
\usepackage{amsmath,amsthm}
\usepackage{amsfonts,amsmath,amsthm, amssymb, mathrsfs}
\usepackage{float}
\usepackage{euscript,eufrak,verbatim}
\usepackage{graphicx}
\usepackage[usenames]{color}
\usepackage[colorlinks,linkcolor=red,anchorcolor=blue,citecolor=blue]{hyperref}
\usepackage{amsmath}
\usepackage{amsthm}
\usepackage[all]{xy}
\usepackage{graphicx} 
\usepackage{amssymb} 
\usepackage{cite}
\newtheorem{thm}{Theorem}

\newtheorem{df}[thm]{Definition}
\newtheorem{lem}[thm]{Lemma}

\newtheorem{ex}[thm]{Example}

\newtheorem{rem}[thm]{Remark}

\def\N{\mathbb{N}}
\def\Z{\mathbb{Z}}
\def\N{\mathbb{N}}
\def\R{\mathbb{R}}

\def\UU{\mathcal{U}}
\def\VV{\mathcal{V}}

\def\XX{\mathcal{X}}

\newcommand{\E}{\mathbb{E}}
\def\UU{\mathcal{U}}

\def\diam{\text{\rm diam}}

\def\Leb{\text{\rm Leb}}
\def\mdim{\text{\rm mdim}}

\def\mdim{\text{\rm mdim}}
\def\over{\overline { \rm mdim}_{M}(T,\XX,d)}
\def\under{\underline{ \rm mdim}_{M}(T,\XX,d)}

\def\logf{\log\frac{1}{\epsilon}}
\makeatletter 
\@addtoreset{equation}{section}

\newcommand{\Rmnum}[1]{\expandafter\@slowromancap\romannumeral #1@}
\makeatother

\ifCLASSINFOpdf
\else
\fi
\begin{document}

\title{Mean dimension and rate-distortion function revisited}

\author{Rui~Yang
      
\thanks{Manuscript received;  revised.}
\thanks{1.School of Mathematics, Northwest University, Xi'an, Shaanxi, 710127, P.R. China; 2. College of Mathematics and Statistics, Center of Mathematics, Chongqing University, Chongqing 401331, P.R. China (e-mail:zkyangrui2015@163.com.) }
}

\markboth{Journal of \LaTeX\ Class Files,~Vol.~14, No.~8, August~2015}%
{Shell \MakeLowercase{\textit{et al.}}: Bare Demo of IEEEtran.cls for IEEE Journals}

\maketitle

\begin{abstract}
Around the mean dimensions and  rate-distortion functions, using some tools from local entropy theory this paper establishes the following main results:

$(1)$  We prove that for non-ergodic measures associated with almost sure processes, the mean R\'enyi information dimension coincides with the information dimension rate. This answers a question posed by  Gutman and \'Spiewak (in Around the variational principle for metric mean dimension, \emph{Studia Math.} \textbf{261}(2021)  345-360).

$(2)$ We introduce three types of rate-distortion entropies and establish their relation with  Kolmogorov-Sinai entropy.

$(3)$ We show that  for systems with the marker property, if the mean dimension is finite, then   the supremum  in  Lindenstrauss-Tsukamoto's double variational principle can be taken over the set of ergodic measures. Additionally, the double variational principle holds for various other  measure-theoretic $\epsilon$-entropies.
\end{abstract}


\IEEEpeerreviewmaketitle

\section{Introduction}

\IEEEPARstart{I}{n} this paper, a topological dynamical system (TDS for short) $(X,T)$ always means that $X$ is a compact metrizable topological space, and $T:X\rightarrow X$ is a homeomorphism. The set of compatible metrics on $X$ is denoted by $\mathcal{D}(X)$. The Borel probability space $M(X)$ is endowed with the weak$^{*}$-topology.  Let  $M(X,T)$, $E(X,T)$  denote the sets of $T$-invariant, and $T$-ergodic   Borel probability measures on $X$, respectively.

Based on the  concept of   uncertainty from information theory,  Kolmogorov-Sinai entropy (or known as measure-theoretic entropy) for   measure-preserving systems  was introduced by  Kolmogorov \cite{kol58} and  Sinai \cite{s59}. Later,  Adler,  Konheim and  McAndrew \cite{akm65} introduced its topological analogue, called topological entropy, for topological dynamical systems. These two concepts are linked by the classical variational principle:
\begin{align*}
	h_{top}(T,X)=\sup_{\mu \in M(X,T)}{h_{\mu}(T)}
	=\sup_{\mu \in E(X,T)}{h_{\mu}(T)},
\end{align*}
where $h_{top}(T,X)$ denotes the topological entropy of $(X,T)$,  and $h_{\mu}(T)$ denotes the  Kolmogorov-Sinai entropy of $\mu$. The variational principle of topological entropy allows us to invoke the techniques from ergodic theory to  study the topological dynamical systems. It plays a crucial role in dimension theory  and chaotic theory of  dynamical systems. 

Since topological entropy fails to characterize dynamics in systems with infinite topological entropy, several refined  entropy-like quantities have been introduced from a quantitative geometric perspective to describe the dynamical behavior of infinite entropy systems.   In 1999, Gromov \cite{gromov} proposed a new topological invariant called mean dimension, which has  found applications in embedding problems of dynamical systems \cite{l99,lw00,g15,glt16,g17,gt20}. To study mean dimension in the context of infinite entropy, Lindenstrauss and Weiss \cite{lw00} introduced the concept of metric mean dimension and showed that  metric mean dimension is an upper bound of mean dimension.   To establish the variational principles for mean dimensions,  Lindenstrauss  and Tsukamoto  \cite{lt18}  introduced the  $L^p$ and $L^\infty$ rate-distortion functions for invariant measures, and  injected ergodic theoretic ideas into mean dimension theory by establishing  the  variational principles for metric mean dimension, i.e.,
\begin{align*}
	\overline{\rm {mdim}}_{M}(T,X,d)=\limsup_{\epsilon \to 0}\frac{1}{\logf}\sup_{\mu \in M(X,T)}R_{\mu,L^{\infty}}(\epsilon).
\end{align*}
Furthermore, if $(X,d)$ has the tame growth of covering numbers\footnote[1]{A  metric $d$ on a compact metric space  $X $  is said to  have the \emph{tame growth of covering numbers} if 
	$\lim_{\epsilon \to 0}\epsilon^{\theta}\log r(X,d,\epsilon)=0$ for each $\theta>0$, where $r(X,d,\epsilon)$  denotes  the smallest number of open balls  $B_d(x,\epsilon)$ needed to cover $X$.
	
	 This definition does not involve the dynamics, and it is proved that every compact  metrizable  space admits a metric satisfying this  condition  \cite[p.22-p.23]{lt19}.}, then   for   $p\in[1,\infty)$, 
\begin{align*}
	\overline{\rm {mdim}}_{M}(T,X,d)=\limsup_{\epsilon \to 0}\frac{1}{\logf} \sup_{\mu \in M(X,T)}R_{\mu,L^p}(\epsilon),
\end{align*}
where  $\overline{\rm mdim}_M(T,X,d)$ denotes the upper metric mean dimension of $(X,d,T)$, and  $R_{\mu,L^p}(\epsilon),R_{\mu,L^{\infty}}(\epsilon)$ are  the $L^p$  and  $L^{\infty}$ rate-distortion functions of $\mu$, respectively.

The authors \cite{gs21,ycz23} further revealed that the  suprema can be taken  over the set of ergodic measures. In \cite{gs21}, Gutman and \'Spiewak introduced the mean  R\'enyi information dimension for invariant measures of a TDS. Furthermore, for the Hilbert cube $[0,1]^{\mathbb{Z}}$ with the left shift $\sigma$,  they proved its coincidence to the information dimension rate for ergodic measures, and posed the question of whether this equivalence holds for non-ergodic measures. We answer affirmatively  in the following theorem.

\begin{thm} \label{thm 1.1}
	Let  $([0,1]^{\mathbb{Z}}, \sigma)$ be a TDS with the metric $d^{\mathbb{Z}}$. Then for every $\mu\in M([0,1]^{\mathbb{Z}}, \sigma)$,
	\begin{align*}
		{\underline{\rm MRID}}([0,1]^{\mathbb{Z}},\sigma,d^{\mathbb{Z}},\mu)&=\underline{d}(\mu),\\
		{\overline{\rm MRID}}([0,1]^{\mathbb{Z}},\sigma,d^{\mathbb{Z}},\mu)&=\overline{d}(\mu),
	\end{align*}
	where  $d^{\mathbb{Z}}(x,y)=\sum_{n\in \mathbb{Z}}\frac{|x_n-y_n|}{2^{|n|}}$; $	{\underline{\rm MRID}}([0,1]^{\mathbb{Z}},\sigma,d^{\mathbb{Z}},\mu)$ and $	{\overline{\rm MRID}}([0,1]^{\mathbb{Z}},\sigma,d^{\mathbb{Z}},\mu)$  respectively denote the  lower and upper mean  R\'enyi information dimensions of $\mu$; $\underline{d}(\mu)$ and $\overline{d}(\mu)$  respectively denote the lower and upper  information dimension rates of $\mu$, see  Subsection  \ref{sub 3.1} for their precise definitions.
\end{thm}

Although the  divergent rate  $$\limsup_{\epsilon \to 0}\frac{h_{\mu}(T,\epsilon)}{\logf}$$ of rate-distortion functions (i.e., the $L^p$ and $L^\infty$ rate-distortion dimensions)  has been studied in \cite{lt18,lt19,gs,w21},   less attention has been devoted to establishing the precise relation between $\lim_{\epsilon \to 0}h_{\mu}(T,\epsilon)$ and   Kolmogorov-Sinai entropy of $\mu$. We introduce three types of rate-distortion entropies and characterize their relation using Kolmogorov-Sinai entropy in the following theorem.

\begin{thm} \label{thm 1.2}
	Let  $(X,T)$ be a TDS.   Then
	
	$(1)$ for every $\mu \in E(X,T)$ and $p\geq 1$, $$h_{\mu,L^p}(T)=h_{\mu,L^\infty}(T)=h_{\mu,B}(T)=h_{\mu}(T);$$
	
	$(2)$ if the system admits the $g$-almost product property, then  for all $\mu \in M(X,T)$ and $p\geq 1$, $$h_{\mu,L^p}(T)=h_{\mu,L^\infty}(T)=h_{\mu,B}(T)=h_{\mu}(T),$$ 
	where $h_{\mu,L^p}(T),  h_{\mu,L^\infty}(T),$ and $h_{\mu,B}(T)$ respectively  denote the $L^P$, $L^{\infty}$ and Bowen rate-distortion entropy of $\mu$, see   Subsection \ref{sub 3.2} for their precise definitions.
\end{thm}

To connect mean dimension theory with ergodic theory, for systems with the marker property Lindenstrauss and Tsukamoto \cite{lt19} established a double variational principle for mean dimension in terms of $L^1$ rate-distortion dimension. Under certain conditions, we prove that the supremum in the double variational principle can be restricted to the set of ergodic measures and that the double variational principle holds for other types of measure-theoretic 
$\epsilon$-entropies studied in \cite{gs21, shi, ycz25}.

\begin{thm}\label{thm 1.3}
	Let $(X,T)$ be a TDS admitting the marker property. If $\mdim(T,X)<\infty$, then  for every $h_{\mu}(T,\epsilon)\in \mathcal{E}\cup \{R_{\mu, L^p}(\epsilon)\}$, 
	\begin{align*}
		\mdim(T,X)&=\min_{d\in \mathcal{D}^{'}(X)} \sup_{\mu \in E(X,T)}\{ \limsup_{\epsilon \to 0}\frac{1}{\logf} h_{\mu}(T,\epsilon)\}\\
		&=\min_{d\in \mathcal{D}^{'}(X)} \sup_{\mu \in M(X,T)}\{ \limsup_{\epsilon \to 0}\frac{1}{\logf} h_{\mu}(T,\epsilon)\},
	\end{align*}
	where  $\mathcal{D}^{'}(X)=\{d\in \mathcal{D}(X): \overline{\rm {mdim}}_{M}(T,X,d)<\infty\}$, and $\mdim(T,X)$ denotes the mean dimension of $(X,T)$; the  measure-theoretic $\epsilon$-entropy $h_{\mu}(T,\epsilon)$  is chosen from the candidate set
	
	$$
	\mathcal{E}= \left\{
	\begin{gathered}
		R_{\mu,L^\infty}(\epsilon),\ 
		\inf_{\substack{\diam(\alpha)\leq\epsilon \\ \alpha\in\mathcal{P}_X}} h_\mu(T,\alpha),\ 
		\inf_{\substack{\diam(\mathcal{U})\leq\epsilon \\ \mathcal{U}\in\mathcal{C}_X^o}} h_\mu(T,\mathcal{U}), \\
		\underline{h}_\mu^K(T,\epsilon,\delta),\ 
		\overline{h}_\mu^K(T,\epsilon,\delta),\ 
		\underline{h}_\mu^K(T,\epsilon), \\
		\overline{h}_\mu^K(T,\epsilon),\ 
		\overline{h}_\mu^{BK}(T,\epsilon),\ 
		{PS}_\mu(T,\epsilon)
	\end{gathered}
	\right\},
	$$ 
	see Subsection \ref{sub 2.3} for their precise definitions.
\end{thm}

The organization of this paper is as follows. In section \ref{sec 2},  we review the definitions of  metric mean dimensions in both topological and measure-theoretic  situations.   In section \ref{sec 3},  we  prove the main Theorems \ref{thm 1.1}-\ref{thm 1.3}.

\section{Preliminaries}\label{sec 2}

In this section, we recall the definitions of metric mean dimension and its measure-theoretic counterpart for invariant measures, which have been used to characterize systems with infinite topological entropy.

\subsection{Mean dimension of dynamical systems}
Let  $(X,T)$ be a TDS. Denote by $\mathcal{C}_{X}$  the collection of  the covers of  $X$ consisting of the Borel measurable sets of $X$.  Let $\mathcal{P}_{X}$ denote the  collection of  finite Borel  (measurable) partitions  of $X$, and let  $ \mathcal{C}_{X}^o$ denote the   collection of finite open covers  of $X$.  The \emph{join} of two covers $\alpha, \beta \in \mathcal{C}_{X}$ is defined by  $$\alpha \vee \beta:=\{A\cap B: A\in \alpha, B \in \beta\}.$$

The $n$-th join  of $\alpha \in \mathcal{C}_X$, denoted by $\alpha^n$, is  the join of $\alpha$ and its preimage covers $\{ T^{-1}\alpha,...,T^{-(n-1)}\alpha\}.$  We  say that $\alpha$ refines $\beta$,  denoted by $\alpha\succ \beta$,  if each element of $\alpha$ is contained  in some element of $\beta$.

Given  $\VV \in \mathcal{C}_{X}^o$, let ${\rm ord}(\VV)=\max_{x\in X}\limits\sum_{V\in \VV}\limits \chi_V(x)-1$. 
For $\UU\in \mathcal{C}_X^o$,  we define  the \emph{order} of $\UU$ as $$\mathcal D(\UU)=\min_{\VV \succ \UU}\limits {\rm ord}(\VV),$$ 
where $\VV$ ranges over  all finite open covers of $X$  refining $\UU$.

\begin{df}
	The {mean dimension} of $(X,T)$  {\rm \cite{gromov,lw00}} is defined by 
	$${\rm mdim} (T,X)=\sup_{\UU\in \mathcal{C}_X^o}\limits \lim_{n\to \infty}\frac{\mathcal D(\UU^n)}{n},$$
	where the  limit exists, because    the sequence $\{\mathcal{D}(\UU^{n})\}_n$ is  sub-additive {\rm \cite[Corollary 2.5]{lw00}}, that is, $ \mathcal{D}(\UU^{n+m})\leq  \mathcal{D}(\UU^{n}) + \mathcal{D}(\UU^{m})$ for all $n,m \geq 1$.
\end{df}

\subsection{Metric mean dimension  of dynamical systems}
Let $(X,d)$ be a compact metric space and $Z$ be a non-empty subset of $X$. Given $\epsilon>0$, a set $E \subset Z$ is a  \emph{$(d,\epsilon)$-spanning set} of $Z$ if for any $x\in Z$, there exists $y \in E$ such that $d(x,y)<\epsilon$. Denote  the  smallest cardinality of $(d,\epsilon)$-spanning sets of $Z$ by  $r(Z,d,\epsilon)$. A subset $F\subset Z$ is a  \emph{$(d,\epsilon)$-separated set} of $Z$ if for any distinct $x,y \in F$, one has $d(x,y)\geq \epsilon$. Denote  the largest cardinality of $(d,\epsilon)$-separated sets of $Z$ by  $s(Z,d,\epsilon)$. 

Let $(X,d, T)$ be a TDS.
The $n$-th Bowen metric  on $X$ is defined by
$$d_n(x,y)=\max_{0\leq j\leq n-1}d(T^jx,T^jy).$$
Then the Bowen ball of $x$ with radius $\epsilon$ in the  metric $d_n$ is given by 
$$B_n(x,\epsilon)=\{y\in X:d_n(x,y)<\epsilon\}.$$ We define the topological $\epsilon$-entropy of $X$ with respect to $T$ as
$$h_{top}(T,X,d,\epsilon)=\limsup_{n\to \infty} \frac{1}{n} \log s(X,d_n,\epsilon).$$
The \emph{topological entropy of $X$} with respect to $T$ is defined by
$$h_{top}(T,X)=\lim_{\epsilon \to 0}h_{top}(T,X,d,\epsilon)=\sup_{\epsilon > 0}h_{top}(T,X,d,\epsilon).$$

The following  definition is  a dynamical analogue of box  dimension in fractal geometry, which was introduced by Lindenstrauss and Weiss \cite{lw00}.

\begin{df}
	The {upper and lower metric mean dimensions}  of $(X,d,T)$ are respectively defined by 
	\begin{align*}
		\overline{\rm mdim}_M(T,X,d)&= \limsup_{\epsilon\to 0}\frac{h_{top}(T,X,d,\epsilon)}{\log \frac{1}{\epsilon}},\\
		\underline{\rm mdim}_M(T,X,d)&= \liminf_{\epsilon\to 0}\frac{h_{top}(T,X,d,\epsilon)}{\log \frac{1}{\epsilon}}.
	\end{align*}
\end{df}

We define the common value ${\rm mdim}_M(T,X,d)$ as  the metric mean dimension of $(X,T)$ if  $\overline{\rm mdim}_M(T,X,d)=\under$.  Unlike the mean dimension, the values of metric mean dimension  depend on the compatible metrics on $X$. Besides, it is easy to see that  any  dynamical system  with finite  topological entropy has zero metric mean dimension, and hence   metric mean dimension is a useful topological quantity to investigate the topological complexity of infinite entropy systems.

\subsection{Measure-theoretic metric mean dimension of invariant measures}\label{sub 2.3}

It is well-known that the measure-theoretic entropy of invariant measures admits several equivalent definitions. For instance, it can be defined by  finite partitions, finite open covers, Brin-Katok entropy, Katok entropy, and other approaches.

Inspired by the definition of metric mean dimension, the authors \cite{ycz25} introduced  measure-theoretic metric mean dimension using measure-theoretic $\epsilon$-entropy, and showed that the divergent rates of these  measure-theoretic $\epsilon$-entropies  of ergodic measures coincide. In the following, we continue to recall several candidates for measure-theoretic $\epsilon$-entropy in \cite{lt18,gs,shi,ycz25}.

\subsubsection{Rate distortion theory}\label{sub 2.2.1}

The  definitions of rate-distortion functions  are due to Lindenstrauss and Tsukamoto \cite{lt18}. A nice and comprehensive introduction  of the interplay between  the rate-distortion dimension theory and the information  theory  can refer to the monographs \cite{ct06,gra 11}. Here, we only  briefly  introduce some relevant concepts for our proofs.

Let $(X,d,T)$ be a TDS, and let $(\Omega,\mathbb{P})$ be a  probability space and $\mathcal{M}$, $\mathcal{N}$ be  two measurable spaces. Suppose that two measurable maps $\xi:\Omega \rightarrow \mathcal{M}$ and $\eta: \Omega \rightarrow \mathcal{N}$ are given. We define the $\emph{mutual information}$ $I(\xi;\eta)$  of $\xi$ and $\eta$ as the supremum of 
\begin{align*}
	\sum_{1\leq m\leq  M,\atop 1\leq n \leq  N}\mathbb{P}(\xi \in P_m,\eta \in Q_n)\log \frac{\mathbb{P}(\xi \in P_m,\eta \in Q_n)}{\mathbb{P}(\xi \in P_m)\mathbb{P}(\eta \in Q_n)},
\end{align*}
where $\{P_1,...,P_M\}$ and $\{Q_1,...,Q_N\}$  are the partitions of $\mathcal{M}$  and $\mathcal{N}$, respectively. Here, we use  the convention that $0\log \frac{0}{a}=0$ for all $a\geq0$.

A measurable map  $\xi: \Omega \rightarrow \mathcal{M}$ with finitely many images  naturally associates a  finite partition on $\Omega$ via $\xi$, i.e.,  the $\emph{preimage partition}$ of $\Omega$; in this case,  by $H(\xi)$ we denote the entropy of  $\xi$. If  both $\mathcal{M}$ and $\mathcal{N}$ are   finite sets, we can rewrite $I(\xi;\eta)$ as
\begin{align*}
	&\sum_{x\in \mathcal{M},y\in \mathcal{N}}\mathbb{P}(\xi =x,\eta =y)\log \frac{\mathbb{P}(\xi =x,\eta =y)}{\mathbb{P}(\xi =x)\mathbb{P}(\eta =y)}\\
	&=H(\xi)-H(\xi|\eta)=H(\xi)+H(\eta)-H(\xi\vee \eta),
\end{align*}
where $H(\xi|\eta)$ is  the conditional entropy of $\xi$ given  $\eta$. 

The value $I(\xi;\eta)$ is non-negative.  It suggests that it is the name that gives the total information  amount shared by both  the random variables $\xi$ and $\eta$.

Let $\epsilon >0$ and $1\leq p<\infty$.  Given $\mu \in M(X,T)$, we define the  $\emph{$L^{p}$ rate-distortion function}$  $R_{\mu,L^p}(\epsilon)$  of $\mu$ as the infimum of
$$\frac{I(\xi;\eta)}{n},$$
where $n$ ranges over  all natural numbers, and $\xi$ and $\eta=(\eta_0,...,\eta_{n-1})$  are random variables  defined on some probability space $(\Omega, \mathbb{P})$ such that
\begin{enumerate}
	\item $\xi$ takes values in $X$, and its law is given by $\mu$.
	\item Each $\eta_k$ takes  values in $X$ and 
	\begin{align*}
		\mathbb{E}\left(\frac{1}{n}\sum_{k=0}^{n-1}d(T^k\xi,\eta_k)^p\right)<\epsilon^p,
	\end{align*}
\end{enumerate} 
where $\mathbb{E}(\cdot)$  is  the  usual expectation w.r.t.  $\mathbb{P}$.

Let $s>0$. We define  $R_{\mu,L^\infty}(\epsilon,s)$  as the infimum of
$$\frac{I(\xi;\eta)}{n},$$
where $n$  ranges over  all natural numbers, and $\xi$ and $\eta=(\eta_0,...,\eta_{n-1})$  are random variables  defined on some probability space $(\Omega, \mathbb{P})$ such that
\begin{enumerate}
	\item $\xi$ takes values in $X$, and its law is given by $\mu$.
	\item Each $\eta_k$ takes  values in $X$ and 
	\begin{align*}
		\mathbb{E}\left( \text{the number of } 0\leq k\leq n-1~\text{with}~d(T^k\xi,\eta_k)\geq \epsilon\right)<sn.
	\end{align*}
	
\end{enumerate} 
We set $R_{\mu,L^{\infty}}(\epsilon)=\lim_{s\to 0}\limits R_{\mu,L^{\infty}}(\epsilon,s)$, and call  $R_{\mu,L^{\infty}}(\epsilon)$ the  $\emph{$L^\infty$ rate-distortion function}$ of $\mu$.


The \emph{upper $L^p$ and $L^{\infty}$ rate-distortion dimensions of $\mu$}  are respectively defined by 
\begin{align*}
	\overline{\rm {rdim}}_{L^p}(X,T,d,\mu)&=\limsup_{\epsilon \to 0}\frac{R_{\mu,L^p}(\epsilon)}{\logf},\\
	\overline{\rm {rdim}}_{L^\infty}(X,T,d,\mu)&=\limsup_{\epsilon \to 0}\frac{R_{\mu,L^\infty}(\epsilon)}{\logf}.
\end{align*}

One can similarly define  lower $L^p$ and $L^{\infty}$ rate-distortion dimensions of $\mu$ by  $\liminf_{\epsilon \to 0}$. 
For every $\mu \in M(X,T)$, it always holds that
$$\overline{\rm {rdim}}_{L^p}(X,T,d,\mu)\leq  \overline{\rm {rdim}}_{L^\infty}(X,T,d,\mu)$$ for every$1\leq p<\infty$.

Another common definition  used in information theory is the  modification of  rate-distortion  conditions of the aforementioned rate-distortion functions. Let $(A,d)$ be a compact metric space and $\sigma: A^{\mathbb{Z}} \rightarrow A^{\mathbb{Z}}$ be the left shift given by $\sigma((x_n)_{n\in \mathbb{Z}})=(x_{n+1})_{n\in \mathbb{Z}}$. The projection $\pi_n: A^{\mathbb{Z}}\rightarrow A^{n}$ is given by
$$\pi_n((x_n)_{n\in \mathbb{Z}})=(x_0,...,x_{n-1}).$$ Let $\epsilon >0$ and $1\leq p<\infty$.  Given $\mu \in M(A^{\mathbb{Z}},\sigma)$, we define the  $\emph{$L^{p}$ rate-distortion function}$  $\widetilde{R}_{\mu,L^p}(\epsilon)$  of $\mu$ as the infimum of
$$\frac{I(\xi;\eta)}{n},$$
where $n$ ranges over  all natural numbers, and $\xi=(\xi_0,...,\xi_{n-1})$ and $\eta=(\eta_0,...,\eta_{n-1})$  are random variables  defined on some probability space $(\Omega, \mathbb{P})$ such that
\begin{enumerate}
	\item $\xi$ takes values in $A^{n}$, and its law is given by $(\pi_n)_{*}\mu$.
	\item $\eta$ takes  values in $A^{n}$ and  approximates $\xi$ in the sense that
	\begin{align*}
		\mathbb{E}\left(\frac{1}{n}\sum_{k=0}^{n-1}d(\xi_k,\eta_k)^p\right)<\epsilon^p,
	\end{align*}
\end{enumerate} 
where $\mathbb{E}(\cdot)$  is  the   expectation w.r.t.  $\mathbb{P}$.

The two different definitions can be  connected for  certain dynamical systems. If we consider the TDS $(X,d,T):=(A^{\mathbb{Z}},d^{\mathbb{Z}}, \sigma)$, where  $d^{\mathbb{Z}}(x,y)=\sum_{n\in \mathbb{Z}}\frac{d(x_n,y_n)}{2^{|n|}}$, then  Gutman and  \'Spiewak \cite[Proposition C-B.1]{gs} showed that $\widetilde{R}_{\mu,L^2}(\epsilon)$ and $R_{\mu, L^2}(\epsilon)$  are related by the following inequality: for every $\epsilon >0$ and  $\mu\in M(A^{\mathbb{Z}}, \sigma)$,
\begin{align}\label{inequ 2.1}
	R_{\mu, L^2}(14\epsilon)\leq \widetilde{R}_{\mu, L^2}(\epsilon)\leq R_{\mu, L^2}(\epsilon).
\end{align}

\subsubsection{Kolmogorov-Sinai $\epsilon$-entropy}
Given a cover $\alpha \in \mathcal{C}_{X}$, the \emph{diameter} of $\alpha$ is defined by $\diam (\alpha,d):=\sup_{A\in \alpha}\diam (A,d)$, and  we sometimes drop  the dependence $d$ if the metric  is clear on the underlying space. The \emph{Lebesgue number} of  a finite open cover $\mathcal{U}\in \mathcal{C}_{X}^o$, denoted by $\Leb (\UU)$, is the largest positive number $\delta>0$ such that each  $d$-open ball $B_d(x,\delta)$ of $X$ is contained in some element of $\UU$.

Let $(X,d,T)$ be a TDS.   It may happen that a random variable $\xi$ defined on a probability space takes infinitely many values in $X$. To compute the entropy of such random variables, we extend the notion of measure-theoretic entropy to arbitrary measurable partitions.
Let $\alpha$ be a  Borel measurable partition of $X$, not  necessarily finite, and let  $\mu \in M(X,T)$. The \emph{partition entropy} of $\alpha$ w.r.t. $\mu$ is defined by  
$$H_\mu(\alpha)=\sum_{A\in \alpha}-\mu(A)\log \mu(A),$$  
where the convention obeys $\log=\log_e$ and $0\cdot \infty=0$. 

Partition entropy is  non-decreasing for  finer partitions. More precisely, let $\alpha\succ \beta$. If $H_{\mu}(\alpha)=\infty$, we have $H_\mu(\beta)\leq H_\mu(\alpha)$; if $H_{\mu}(\alpha)<\infty$, then   the atoms of $\alpha$ with positive $\mu$-measure   is at most countable, and  each atom (mod $\mu$)  of $\beta$  with  positive $\mu$-measure  is   the   union of some  atoms of $\alpha$ with   positive $\mu$-measure. This  yields that
\begin{align*}
	H_\mu(\beta)=&\sum_{B\in \beta}-\mu(B)\log \mu(B)\\
	\leq& \sum_{B\in \beta}\sum_{A\subset 
		B, A\in \alpha}-\mu(A)\log \mu(A)
	=
	H_\mu(\alpha).
\end{align*}
Then, if $H_\mu(\alpha^n)=\infty$ for some $n$, we set $h_\mu(T,\alpha):=\infty$; otherwise, we define the  \emph{Kolmogorov-Sinai entropy of $\alpha$ w.r.t. $\mu$} as
$$h_\mu(T,\alpha)=\limsup_{n\to \infty}\frac{1}{n}H_\mu(\alpha^n).$$
The  classical \emph{Kolmogorov-Sinai entropy  of $\mu$} is defined  by $$h_{\mu}(T)=\sup_{\alpha \in \mathcal{P}_X}h_{\mu}(T,\alpha).$$ We define the \emph{Kolmogorov-Sinai $\epsilon$-entropy of $\mu$} as
$$\inf_{\diam  (\alpha) \leq \epsilon, \atop  \alpha \in \mathcal{P}_{X}}h_\mu(T,\alpha).$$

\subsubsection{Brin-Katok's  $\epsilon$-entropy} It is defined  by  a ``local" viewpoint.

Let  $\epsilon >0$ and  $\mu \in {M}(X)$.
We   respectively define   \emph{the upper and lower Brin-Katok local $\epsilon$-entropies of $\mu$} as
\begin{align*}
	\overline{h}_{\mu}^{BK}(T, \epsilon):&=\int \limsup_{n\to \infty}-\frac{\log \mu (B_n(x,\epsilon))}{n}d\mu,\\
	\underline{h}_{\mu}^{BK}(T, \epsilon):&=\int \liminf_{n\to \infty}-\frac{\log \mu (B_n(x,\epsilon))}{n}d\mu.
\end{align*}

For every $\mu \in M(X,T)$, it holds that \cite{bk83} $$\lim_{\epsilon\to 0}\underline{h}_{\mu}^{BK}(T, \epsilon)=\lim_{\epsilon\to 0}\overline{h}_{\mu}^{BK}(T, \epsilon)=h_\mu(T).$$

\subsubsection{Katok's $\epsilon$-entropies}
It is defined  using spanning sets \cite{k80} and   finite open covers \cite{s07}.

Given  $\delta \in (0,1)$, $\epsilon>0$, $n \in \mathbb{N}$ and  $\mu \in  M(X)$, let $R_\mu^\delta(T,n, \epsilon)$ denote  the minimal cardinality of a subset $E$ of $X$ satisfying  $$\mu (\cup_{x\in E}B_n(x,\epsilon))> 1-\delta. $$

We  respectively define  the \emph{upper and lower Katok's $\epsilon$-entropies  of $\mu$} as
\begin{align*}
	\overline{h}_{\mu}^K(T,\epsilon, \delta)&=\limsup_{n\to \infty} \frac{1}{n} \log R_\mu^\delta(T,n, \epsilon),\\
	\underline{h}_{\mu}^K(T,\epsilon, \delta)&=\liminf_{n\to \infty} \frac{1}{n} \log R_\mu^\delta(T,n, \epsilon).
\end{align*}

Notice that  the quantities $\overline{h}_{\mu}^K(T,\epsilon, \delta)$, $\underline{h}_{\mu}^K(T,\epsilon, \delta)$ are  non-decreasing as $\delta$ decreases.  This fact allows us to    define two new \emph{upper and lower Katok's $\epsilon$-entropies  of $\mu$}:
\begin{align*}
	\overline{h}_{\mu}^K(T,\epsilon):=\lim_{\delta \to 0}\overline{h}_{\mu}^K(T, \epsilon, \delta),~~
	\underline{h}_{\mu}^K(T,\epsilon):=\lim_{\delta \to 0}\underline{h}_{\mu}^K(T,\epsilon, \delta).
\end{align*}

If $\mu \in E(X,T)$, Katok \cite{k80} showed  that  for every $\delta \in(0,1)$, one has
$$\lim_{\epsilon \to 0}\overline{h}_{\mu}^K(T,\epsilon, \delta)=\lim_{\epsilon \to 0}\underline{h}_{\mu}^K(T,\epsilon, \delta)=h_{\mu}(T).$$ 

Besides, Katok's entropy of ergodic measures  admits a formulation using  finite open covers.  

Let $\delta \in (0,1)$ and $\mathcal{U} \in \mathcal{C}_X^o$. Given  $\mu \in E(X,T)$,  we define $N_{\mu}(\mathcal{U},\delta)$ as the  minimal cardinality of a subfamily of $\UU$ whose union has $\mu$-measure greater than $1-\delta$. The  \emph{Shapira's entropy} of $\mu$ w.r.t. $\UU$  is defined by 
$$h_\mu^S(\mathcal{U}):=\lim_{n\to \infty}\frac{\log N_{\mu}(\mathcal{U}^n,\delta)}{n},$$
where the limit
exists and is independent of the  choice of $\delta \in (0,1)$ \cite[Theorem 4.2]{s07}.

The  \emph{Shapira's $\epsilon$-entropy} of $\mu$ is defined by
$$\inf_{\diam (\UU)\leq \epsilon, \atop \UU \in \mathcal{C}_{X}^o}h_\mu^S(\mathcal{U}).$$

\subsubsection{Pfister and Sullivan's  $\epsilon$-entropy}
Recall that the weak$^{*}$-topology of $M(X)$ is metrizable by   a  compatible metric
$$D(m,\mu)=\sum_{n=1}^{\infty}\frac{|\int f_n dm-\int f_n d\mu|}{2^n (||f_n||+1)},$$
where $\{f_n\}_{n=1}^{\infty}$ is  a dense  subset of $C(X)$. 
Let $\mu \in M(X)$.  By a subset $F \subset M(X)$ we mean a neighborhood of $\mu$ if $F$ contains a $D$-open ball $B_D(\mu,\gamma_0)$ for some $\gamma_0 >0$. We put $$X_{n,F}=\{x\in X: \frac{1}{n}\sum_{j=0}^{n-1} \delta_{T^{j}(x)} \in F\},$$
where $\delta_x$ is the   Dirac mass at  $x\in X$. 

For $\epsilon >0$, we define the \emph{Pfister and Sullivan's  $\epsilon$-entropy of $\mu$} as
\begin{align*}
	PS_\mu(T,\epsilon)=\inf_{F\ni \mu}\limsup_{n \to \infty}\frac{1}{n}\log s(X_{n,F},d_n,\epsilon),
\end{align*}
where the infimum  ranges  over  all neighborhoods  $F$  in $M(X)$ of $\mu$.

In \cite{ps07}, Pfister and Sullivan   proved that  for every $\mu \in E(X,T)$, $$h_{\mu}(T)=\lim_{\epsilon \to 0}\limits PS_\mu(T,\epsilon).$$ 

Now we are in a position to  collect some standard facts  involving the relations of these measure-theoretic $\epsilon$-entropies.

\begin{lem}\label{lem 2.1}
	Let $(X,T)$ be a TDS with a metric $d \in \mathcal{D}(X)$. Then the following statements hold:
	
	$(1)$ For every $\mu \in E(X,T)$ and  $\UU \in \mathcal{C}_{X}^o$, one has $$h_\mu^S(\mathcal{U})=h_\mu(T,\mathcal{U}),$$
	where $h_\mu(T,\mathcal{U}):=\inf_{\alpha \succ \mathcal{U},\alpha \in \mathcal{P}_{X}}h_{\mu}(T,\alpha)$ is called the local measure-theoretic entropy of  $\UU$ w.r.t. $\mu$.
	
	$(2)$ Fix $\UU \in \mathcal{C}_{X}^o$. The local entropy map  $\mu \in M(X,T)\mapsto h_\mu(T,\mathcal{U})$ is affine and upper semi-continuous. Assume that $\mu \in M(X,T)$ and  $\mu=\int_{E(X,T)}m d\tau(m)$ is the  ergodic decomposition of $\mu$. Then
	$$h_\mu(T,\mathcal{U})=\int_{E(X,T)}h_m(T,\mathcal{U})d\tau(m).$$
	
	$(3)$\footnote[2]{A corresponding  statement for the action of amenable groups is given in \cite[Theorem 3.17]{y25}.}  For every $\mu \in E(X,T)$, the upper limit 
	$$\limsup_{\epsilon \to 0}\frac{h_{\mu}(T,\epsilon)}{\logf}$$
	is independent of  the candidate  $h_{\mu}(T,\epsilon)$ chosen from  the candidate set 
	$\mathcal{E}$.
	
	Besides,  for any $h_{\mu}(T,\epsilon)\in \mathcal{E}$, it satisfies the variational principles: 
	\begin{align*}
		\overline{\rm mdim}_M(T,X,d)&=\limsup_{\epsilon \to 0}\frac{1}{\logf}\sup_{\mu \in E(X,T)}h_{\mu}(T,\epsilon)\\
		&=\limsup_{\epsilon \to 0}\frac{1}{\logf}\sup_{\mu \in M(X,T)}h_{\mu}(T,\epsilon).
	\end{align*}
	The corresponding results are  also valid for the case of lower limits.
\end{lem}

\begin{proof}
	(1) is given by Shapira \cite[Theorem 4.4]{s07};
	(2) follows from \cite[Proposition 3.8, Theorem 3.13]{hyz11}; (3) is due to \cite[Theorems 1.1-1.3]{ycz25}.
\end{proof}

\begin{rem}
	If we let $h_\mu(T,\varepsilon) \in \{\overline{h}_\mu^K(T,\varepsilon,\delta), \underline{h}_\mu^K(T,\varepsilon,\delta)\}$, the statement  of Lemma \ref{lem 2.1} (3) holds for every $\delta \in (0,1)$.
\end{rem}

Finally, we  present an example to   clarify  the  definitions of metric mean dimensions (cf. \cite[E. Example]{lt18} and Lemma \ref{lem 2.1} (3)).
\begin{ex}
	As   in Subsecion  \ref{sub 2.2.1} above, let $A=[0,1]$ and $d=|\cdot|$  be the standard  Euclidean metric, and let  $\mu=\mathcal{L}^{\otimes\mathbb{Z}}$ be the product measure on $[0,1]^{\mathbb{Z}}$, where $\mathcal{L}$ is the Lebesgue  measure on $[0,1]$. Then  for every $h_{\mu}(T,\epsilon)\in \mathcal{E}$,
	$${\rm mdim}(\sigma, [0,1]^{\mathbb{Z}})={\rm mdim}_M(\sigma,[0,1]^{\mathbb{Z}},d^{\mathbb{Z}})=1=\lim_{\epsilon \to 0}\frac{h_{\mu}(T,\epsilon)}{\logf},$$
	where $d^{\mathbb{Z}}(x,y)=\sum_{n\in \mathbb{Z}}\frac{|x_n-y_n|}{2^{|n|}}$.
\end{ex}

\section{Proofs of main results}\label{sec 3}

\subsection{An answer to  Gutman-\'Spiewak's open question}\label{sub 3.1}

Using the notions introduced in the previous subsections,  in this subsection we prove  Theorem \ref{thm 1.1}.

To this end, we first review the  precise definitions of  the  mean R\'enyi information dimension  and  information dimension rate for stationary stochastic processes.

Inspired  by  the concept of  \emph{R\'enyi information dimension}, Gutman and \'Spiewak \cite{gs21} introduced  the 
\emph{lower and upper mean  R\'enyi information dimensions}  of $\mu \in M(X,T)$:
\begin{align*}
	{\underline{\rm MRID}}(X,T,d,\mu)&=\liminf_{\epsilon \to 0}\frac{1}{\logf}\inf_{\diam  (\alpha) \leq \epsilon}h_\mu(T,\alpha),\\
	{\overline{\rm MRID}}(X,T,d,\mu)&=\limsup_{\epsilon \to 0}\frac{1}{\logf}\inf_{\diam  (\alpha) \leq \epsilon}h_\mu(T,\alpha),
\end{align*}
where the infimum ranges over all measurable  partitions of $X$ with diameter at most $\epsilon$.

Besides, Geiger and Koch  considered an analogous definition  for stationary stochastic processes taking values in $[0,1]$.  Let $([0,1]^{\mathbb{Z}}, \sigma)$ be a TDS, where the product topology of $[0,1]^{\mathbb{Z}}$ is metrizable by the metric $$d^{\mathbb{Z}}((x_n)_{n\in \mathbb{Z}}, (y_n)_{n\in \mathbb{Z}})=\sum_{n\in \mathbb{Z}}\frac{|x_n-y_n|}{2^{|n|}},$$
and $\sigma: [0,1]^{\mathbb{Z}} \rightarrow [0,1]^{\mathbb{Z}}$ is the left shift map. Given $\mu\in M([0,1]^{\mathbb{Z}}, \sigma)$, the  \emph{lower and upper  information dimension rates of $\mu$} are respectively defined by
\begin{align*}
	\underline{d}(\mu)=\liminf_{m \to \infty}\frac{h_{\mu}(\sigma, \alpha_m)}{\log m},~~
	\overline{d}(\mu)=\limsup_{m \to \infty}\frac{h_{\mu}(\sigma, \alpha_m)}{\log m},
\end{align*}
where $\alpha_m:=\pi^{-1}(\{[\frac{i}{m}, \frac{i+1}{m})\cap [0,1]: i\in \mathbb{Z}\})$ is a finite partition of $[0,1]^{\mathbb{Z}}$, and $\pi$ is the projection  assigning each point  in $[0,1]^{\mathbb{Z}}$ to its $0$-coordinate.

Geiger and Koch (\cite[Theorem 1]{gk17} and \cite[Theorem 18]{gk19}) proved  that for every $\mu\in M([0,1]^{\mathbb{Z}}, \sigma)$, one has
\begin{align*}
	\underline{d}(\mu)=\liminf_{\epsilon \to 0}\frac{\widetilde{R}_{\mu, L^2}(\epsilon)}{\logf},~~\overline{d}(\mu)=\limsup_{\epsilon \to 0}\frac{\widetilde{R}_{\mu, L^2}(\epsilon)}{\logf}.
\end{align*}

Therefore,  by (\ref{inequ 2.1}) we conclude that for every  $\mu\in M([0,1]^{\mathbb{Z}}, \sigma)$,
\begin{align}\label{inequ 3.1}
	\underline{d}(\mu)&=\underline{\rm {rdim}}_{L^2}([0,1]^{\mathbb{Z}},\sigma,d^{\mathbb{Z}},\mu),  \nonumber\\
	\overline{d}(\mu)&=\overline{\rm {rdim}}_{L^2}([0,1]^{\mathbb{Z}},\sigma,d^{\mathbb{Z}},\mu).
\end{align}

Later, Gutman and \'Spiewak
\cite[Proposition 4.2]{gs21} showed that for  every  $\mu\in E([0,1]^{\mathbb{Z}}, \sigma)$, one has
\begin{align*}
	{\underline{\rm MRID}}([0,1]^{\mathbb{Z}},\sigma,d^{\mathbb{Z}},\mu)&=\underline{d}(\mu),\\ {\overline{\rm MRID}}([0,1]^{\mathbb{Z}},\sigma,d^{\mathbb{Z}},\mu)&=\overline{d}(\mu),
\end{align*}
and posed a question \cite[Problem 2]{gs21}  whether it holds for non-ergodic measures on  $[0,1]^{\mathbb{Z}}$.

\begin{lem}\label{lem 3.1}
	Let $(X,T)$ be a TDS with a metric $d \in \mathcal{D}(X)$ and $\mu \in M(X,T)$. Then  for every $p\in [1,\infty)$,
	\begin{align*}
		\underline{\rm {rdim}}_{L^p}(X,T,d,\mu)&\leq {\underline{\rm MRID}}(X,T,d,\mu),\\
		\overline{\rm {rdim}}_{L^p}(X,T,d,\mu)&\leq {\overline{\rm MRID}}(X,T,d,\mu).
	\end{align*}
\end{lem}

\begin{proof}
	Fix $\mu \in M(X,T)$. We divide the proof into  the following two steps:
	
	\emph{Step 1.} We show  the inequality: $$	\inf_{\diam  (\alpha) \leq \epsilon, \atop \alpha \in \mathcal{P}_{X}}\limits h_\mu(T,\alpha) \leq \inf_{\diam  (\alpha) \leq \frac{\epsilon}{8}} h_\mu(T,\alpha)$$ for every $\epsilon >0$, where  the infimum in the RHS ranges over all measurable  partitions of $X$ with diameter at most $\epsilon$.
	
	Clearly, one has $$\inf_{\diam  (\alpha) \leq \epsilon}h_\mu(T,\alpha)<\infty$$ for every $\epsilon >0$.  Let  $\alpha$ be an uncountable partition of $X$ with diameter  at most $\epsilon$ such that $h_{\mu}(T,\alpha)<\infty$. Then for some sufficiently  $n$, we have $$H_{\mu}(\alpha) \leq  H_{\mu}(\alpha^n)<\infty.$$ 
	Then  $\alpha$  has at most countably many atoms  with  positive $\mu$-measure. By   $\mathcal{F}$  we denote these atoms. Then, by the definition of  Borel $\sigma$-algebra,  $X \backslash \cup \mathcal{F}$ is a zero $\mu$-measure  set.  The compactness of $X$ allows  us to get a  new family $\mathcal{F^{'}}$  consisting  of finitely many  pairwise  disjoint Borel subsets of $X \backslash \cup \mathcal{F}$ with  the property that each   has diameter at most $\epsilon$ and  zero $\mu$-measure. Then we have  $h_\mu(T,\alpha)=h_\mu(T,\mathcal{F}\cup \mathcal{F}^{'})$, and hence $\inf_{\diam  (\alpha) \leq \epsilon}h_\mu(T,\alpha)$ suffices to  take the infimum over those (at most) countable partitions  $\alpha$ of $X$ with diameter  at most $\epsilon$ and $h_{\mu}(T,\alpha)<\infty$.

	Given a finite open cover  $\UU$ of $X$, we define  $$\widetilde{h}_\mu(T,\mathcal{U}):=\inf_{\alpha \succ \mathcal{U}}h_{\mu}(T,\alpha),$$ 
	where the infimum is taken  over all  countable Borel partitions of $X$ refining $\UU$. We
	claim that
	\begin{align}\label{inequ 3.3}
		h_\mu(T,\mathcal{U})=\widetilde{h}_\mu(T,\mathcal{U}).
	\end{align}
	To see this, the inequality $h_\mu(T,\mathcal{U})\geq \widetilde{h}_\mu(T,\mathcal{U})$  is clear. Now if $\alpha$ is a countable partition of $X$ that refines $\mathcal{U}$, then  there exists a finite partition  $\beta$ of $X$ such that $\alpha \succ \beta$ and
	$${h}_\mu(T,\mathcal{U})\leq h_{\mu}(T, \beta)\leq h_{\mu}(T,\alpha).$$
	The arbitrariness of  $\alpha$ implies that $h_\mu(T,\mathcal{U})\leq\widetilde{h}_\mu(T,\mathcal{U})$.
	
	Now  take a countable partition  $\alpha$ of $X$ with $\diam(\alpha)\leq \frac{\epsilon}{8} $  and $h_{\mu}(T,\alpha)<\infty$. Let $\UU \in \mathcal{C}_{X}^o$ with $\diam (\UU) \leq \epsilon$ and $\Leb (\UU) \geq \frac{\epsilon}{4}$\footnote[3]{See \cite[Lemma 3.4]{gs21} for the  existence of such open covers. For instance, consider the family $\UU=\{B_d(x,\frac{\epsilon}{2}):x\in E\}$ of open sets of $X$, where $E$ is a finite $\frac{\epsilon}{4}$-net of $X$.}. Since  each partition $\alpha$ of $X$ refining $\UU$ has diameter  at most $\epsilon$, and  the atoms of each partition  $\alpha $  of $X$ with $\diam  (\alpha) \leq \frac{\epsilon}{8}$  are contained in some element of $\UU$. By  (\ref{inequ 3.3}), this yields that
	\begin{align}\label{equ 2.1}
		\inf_{\diam  (\alpha) \leq \epsilon, \atop \alpha \in \mathcal{P}_{X}}h_\mu(T,\alpha)\leq h_\mu(T,\mathcal{U})= \widetilde{h}_\mu(T,\mathcal{U})\leq  h_\mu(T,\alpha),
	\end{align}
	and hence  finishes step 1.

	\emph{Step 2.} We show the inequality:
	\begin{align}\label{equ 2.2}
		R_{\mu, L^p}(2\epsilon)\leq \inf_{\diam  (\alpha) \leq \epsilon, \atop \alpha \in \mathcal{P}_{X}}h_\mu(T,\alpha)
	\end{align}
	for each $\epsilon>0$ and $p\in [1,\infty)$. 
	
	Fix  $\alpha \in \mathcal{P}_{X}$ with $\diam  (\alpha) \leq \epsilon$.  Let $\xi$ be a random variable taking values in $X$  whose  law  obeys $\mu$. Fix  $n \in \mathbb{N}$. Without loss of generality,  assume that  each atom of $\alpha^n$ is non-empty.  Take arbitrarily a point $x_A\in A$ from the atom $A\in \alpha^n$,  and then define a map  $f: X \rightarrow X$ by assigning each $x\in X$ to  $x_A$ if $x\in A$. Let $\eta={(f(\xi),Tf(\xi),...,T^{n-1}f(\xi))}$ be  another random variable. Then
	\begin{align}\label{inequ 3.6}
		\mathbb{E}\left(\frac{1}{n}\sum_{k=0}^{n-1}d(T^k\xi,\eta_k)^p\right)=&\int_{X} \frac{1}{n}\sum_{k=0}^{n-1}d(T^k(x),T^{k}f(x))^pd\mu(x) \nonumber \\
		\leq& \epsilon^p<(2\epsilon)^p.
	\end{align}
	Therefore,  for every $n\geq 1$ we  obtain that
	$$R_{\mu,L^{p}} (2\epsilon)\leq \frac{I(\xi;\eta)}{n}\leq\frac{H(\eta)}{n}=\frac{H_\mu(\alpha^n)}{n}.$$
	This implies that  $R_{\mu, L^p}(2\epsilon)\leq \inf_{\diam  (\alpha) \leq \epsilon, \atop \alpha \in \mathcal{P}_{X}}\limits h_\mu(T,\alpha)$.

	This completes the proof by steps 1 and 2.
\end{proof}

Now we  affirmatively answer the aforementioned question by  proving  Theorem \ref{thm 1.1}.

\begin{proof}[Proof of  Theorem \ref{thm 1.1}]
	We only prove  $${\underline{\rm MRID}}([0,1]^{\mathbb{Z}},\sigma,d^{\mathbb{Z}},\mu)=\underline{d}(\mu).$$ The same proof  works for $	{\overline{\rm MRID}}([0,1]^{\mathbb{Z}},\sigma,d^{\mathbb{Z}},\mu)=\overline{d}(\mu)$.
	
	Fix $m\geq 1$. Choose sufficiently large $N$ (depending on $m$) such that  the diameter of  the partition $\vee_{j=-N}^{N}\sigma^{-j}\alpha_m$ of $[0,1]^{\mathbb{Z}}$ is bounded  above  by
	$$\diam (\vee_{j=-N}^{N}\sigma^{-j}\alpha_m, d^{\mathbb{Z}})<\frac{3}{m}+(\frac{1}{2})^{N-2}<\frac{4}{m}.$$  It follows that
	\begin{align}\label{inequ 3.2}
		\inf_{\diam  (\alpha) \leq \frac{4}{m}}h_\mu(\sigma,\alpha)\leq  h_\mu(\sigma,\vee_{j=-N}^{N}\sigma^{-j}\alpha_m)= h_\mu(\sigma,\alpha_m).
	\end{align}
	Choose a strictly  increasing subsequence $\{m_k\}_k$ of positive  integers such that
	$$\underline{d}(\mu)=\lim_{k \to \infty}\frac{h_{\mu}(\sigma, \alpha_{m_k})}{\log m_k},$$
	and take  $\epsilon_k=\frac{4}{m_{k}}$ for each $k$. Then $\lim_{k\to \infty}\frac{\log \frac{1}{\epsilon_k}}{\log m_k}=1$ by the choice of $\epsilon_k$. Using (\ref{inequ 3.2}), these arguments enable us to obtain 
	$${\underline{\rm MRID}}([0,1]^{\mathbb{Z}},\sigma,d^{\mathbb{Z}},\mu)\leq \underline{d}(\mu)$$
	for any $\mu \in M([0,1]^{\mathbb{Z}},\sigma)$.
	
	On the other hand,  by (\ref{inequ 3.1}) we have $$\underline{d}(\mu)=\underline{\rm {rdim}}_{L^2}([0,1]^{\mathbb{Z}},\sigma,d^{\mathbb{Z}},\mu).$$ Together with Lemma \ref{lem 3.1},  it implies  that for any $\mu \in M([0,1]^{\mathbb{Z}},\sigma)$,
	\begin{align*}
		\underline{d}(\mu)= \underline{\rm {rdim}}_{L^2}([0,1]^{\mathbb{Z}},\sigma,d^{\mathbb{Z}},\mu)\leq {\underline{\rm MRID}}([0,1]^{\mathbb{Z}},\sigma,d^{\mathbb{Z}},\mu).
	\end{align*}

\end{proof}

\subsection{Linking rate-distortion entropy and  Kolmogorov-Sinai entropy}\label{sub 3.2}

From the perspective of statistical mechanics, it is well-known that the topological entropy (defined by  Bowen metric) can be equivalently expressed in terms of  mean metric,  \(g\)-mistake function, and other dynamical metrics. This result extends to the metric mean dimension of TDSs satisfying the (weak) tame growth of covering numbers \cite[Theorem 1.1]{xcy25}. We aim to determine whether the rate-distortion entropy remains unchanged when certain rate-distortion conditions are imposed. To this end, we introduce three types of rate-distortion entropies under different rate-distortion conditions and prove Theorem \ref{thm 1.2}.

Let $(X,T)$ be a TDS with a metric $d \in \mathcal{D}(X)$ and $\mu \in M(X,T)$. 

\subsubsection{$L^p$  rate-distortion entropy}
Notice that $R_{\mu,L^p}(\epsilon)$, $1\leq p<\infty$,  is defined by  the $n$-th mean metric  on $X$, i.e., $$\overline{d}_n(x,y)=\frac{1}{n}\sum_{j=0}^{n-1}d(T^jx,T^jy),$$ 
and  is non-increasing in $\epsilon$. 

We  define  the  \emph{$L^p$ rate-distortion entropy of $\mu$} as
\begin{align*}
	h_{\mu,L^P}(T)=\lim_{\epsilon  \to 0} R_{\mu,L^p}(\epsilon).
\end{align*}

\subsubsection{Bowen and $L^{\infty}$ rate-distortion entropies}

As we have done for topological entropy using Bowen metric,  only replacing the condition (2) presented  for $R_{\mu,L^p}(\epsilon)$ by
\begin{align*}
	\mathbb{E}\left(\max_{0\leq k <n}d(T^k\xi,\eta_k)\right)<\epsilon,
\end{align*}
we  similarly define the Bowen   rate-distortion function $R_{\mu,B}(\epsilon)$  of $\mu$ and the  \emph{Bowen   rate-distortion entropy of $\mu$} as
$$h_{\mu,B}(T)=\lim_{\epsilon  \to 0} R_{\mu,B}(\epsilon).$$

Actually, Lindenstrauss and Tsukamoto established  variational principles for the metric mean dimension (defined by the  mean metric) in terms of \(L^p\) (\(1 \leq p < \infty\)) rate-distortion function \cite[Corollary 33]{lt18}, and another variational principles for the  metric mean dimension  (defined by the  Bowen metric) in terms of \(L^\infty\) rate-distortion function \cite[Theorem 9]{lt18}, without considering the Bowen metric as a distortion condition. This is the reason why we introduce the Bowen rate-distortion entropy of invariant measures. Furthermore,   it is worth noting that the distortion condition involving \(R_{\mu,L^\infty}(\epsilon,r)\) is similar to the definition of an \(r\)-Bowen ball \cite{rhlz11}, i.e.,
$$B_n(x,\epsilon,r):=\{y\in X: \frac{\#\{0\leq j<n:d(T^jx,T^jy)<\epsilon\}}{n}>1-r\},$$
where \(r\in (0,1)\).

Recall that $R_{\mu,L^{\infty}}(\epsilon)=\lim_{r \to 0}R_{\mu,L^{\infty}}(\epsilon,r)$. When $r$ is fixed, it is not straightforward to check the monotonicity of $R_{\mu,L^\infty}(\epsilon)$ in $\epsilon$ from the distortion condition:
$$\mathbb{E}\left( \text{the number of } 0\leq k\leq n-1~\text{with}~d(T^k\xi,\eta_k)\geq \epsilon\right)<nr.$$
Since  the distortion condition  can be rewritten as
\begin{align*}
	\mathbb{E}\left( \text{the number of } 0\leq k\leq n-1~\text{with}~d(T^k\xi,\eta_k)< \epsilon\right)>(1-r)n,
\end{align*}
we know that $R_{\mu,L^{\infty}}(\epsilon,r)$ is non-increasing in $\epsilon$ for every fixed $r\in (0,1)$,  which  enables us to  define the  \emph{$L^{\infty}$  rate-distortion entropy of $\mu$} as
$$h_{\mu,L^\infty}(T)=\lim_{\epsilon \to 0} R_{\mu,L^{\infty}}(\epsilon).$$

\subsubsection{Proof of  Theorem \ref{thm 1.2}}
We also recall the   definition of $g$-almost product property introduced by Pfister and Sullivan \cite{ps07}, which is weaker than the specification property and is realized by a $g$-mistake function.
\begin{df}
	A map $g:\mathbb{N}\to\mathbb{N}$ is said to be a mistake function if $g$ is non-decreasing with the properties that for every $n\geq 2$,  $g(n)<n$, and 
	\begin{equation*}
		\lim_{n\to\infty}\frac{g(n)}{n}=0.
	\end{equation*}
	
	Given $x\in X$ and $\epsilon>0$, the $g$-mistake Bowen ball $B_{n}(g;x,\epsilon)$ of $x$  is defined by
	\begin{align*}
		\left\{y\in X: \exists \Lambda \subset \{0,1,...,n-1\}~\text{with}~n-\#(\Lambda)<g(n)~\atop\text{and} ~\max_{j\in \Lambda}d(T^jx,T^jy)<\epsilon \right\}.
	\end{align*}
\end{df}
Compared with the Bowen ball $B_n(x,\epsilon)$, the $g$-mistake ball allows at most $g(n)$ errors for $\epsilon$-approximating the orbit of a point, and as time evolves, the total errors decrease (rapidly) in \(n\). This gives rise to the partial shadowing property.

\begin{df}{\rm \cite[Definition 2.3]{ps07}}
	A TDS  $(X,d,T)$ is said to  have the $g$-almost product property  if there exists a non-increasing function $m:\mathbb{R}^{+}\to\mathbb{N}$ such that for any $k\in\mathbb{N}$, any $x_{1},\cdots,x_{k}\in X$, any positive number $\epsilon_{1},\cdots,\epsilon_{k}$ and any integers $n_{1}\ge m(\epsilon_{1}),\cdots,n_{k}\ge m(\epsilon_{k})$,
	\begin{equation*}
		\bigcap_{j=1}^{k}T^{-M_{j-1}}B_{n_{j}}(g;x_{j},\epsilon_{j})\ne\emptyset,
	\end{equation*}
	where $M_{0}:=0$, $M_{i}:=n_{1}+\cdots +n_{i}$, $i=1,\cdots,k-1$.
\end{df}

The examples  of TDSs with the  $g$-almost  product property include the  full shifts on  any compact metric  space, the topological mixing subshifts of finite type \cite[Proposition 21.2]{Dgs76},  the   topological mixing locally maximal hyperbolic set, and the $\beta$-shifts \cite{ps07}.

\begin{lem}{\rm \cite[Corollary 3.2, Proposition 6.1]{ps07} \label{lem 3.3}}
	Let  $(X,d,T)$ be a TDS.   Then
	
	$(1)$ for every $\mu \in E(X,T)$, $\lim\limits_{\epsilon \to 0}PS_\mu(T,\epsilon)=h_{\mu}(T)$;
	
	$(2)$ if the system admits the $g$-almost product property, then  $\lim\limits_{\epsilon \to 0}PS_\mu(T,\epsilon)=h_{\mu}(T)$ for all $\mu \in M(X,T)$.
\end{lem}

\begin{proof}[Proof of Theorem \ref{thm 1.2}]
	We divide the proof into the following three steps:
	
	\emph{Step 1.}
	For every $\mu \in M(X,T)$, we show the inequality:
	\begin{align}\label{equ 3.7}
		\lim\limits_{\epsilon \to 0}PS_\mu(T,\epsilon)\leq h_{\mu,L^P}(T)
	\end{align}
	for all $1\leq p <\infty$.
	
	Fix $\mu \in M(X,T)$. By H\"{o}lder inequality,  we know that
		$h_{\mu,L^P}(T) \geq h_{\mu,L^1}(T)$ for every $p>1$. Hence, we  prove  (\ref{equ 3.7}) by verifying $\lim\limits_{\epsilon \to 0}PS_\mu(T,\epsilon)\leq h_{\mu,L^1}(T)$.
		
		Using the $n$-th mean metric $\overline{d}_n(x,y)$,  we define Pfister and Sullivan's  $\epsilon$-entropy of $\mu$ in the mean sense as
		$$\overline{PS}_\mu(T,\epsilon)=\inf_{F\ni \mu}\limsup_{n \to \infty}\frac{1}{n}\log r(X_{n,F},\overline{d}_n,\epsilon).$$
		We claim that 
		\begin{align}\label{equ 3.8}
			\lim_{\epsilon \to 0}\overline{PS}_\mu(T,\epsilon)=\lim_{\epsilon \to 0}{PS}_\mu(T,\epsilon).
		\end{align}
		
		Indeed, one has $\lim_{\epsilon \to 0}\overline{PS}_\mu(T,\epsilon)\leq \lim_{\epsilon \to 0}{PS}_\mu(T,\epsilon)$ since $\overline{d}_n\leq d_n$, and every $(d_n,\epsilon)$-separated set of $X_{n,F}$ with the largest cardinality is also a $(\overline{d}_n,\epsilon)$-spanning set of $X_{n,F}$. To get the reverse inequality, we  need to show
		\begin{align}\label{equ 3.9}
			\lim_{\epsilon \to 0}\overline{PS}_\mu(T,\epsilon)\geq\lim_{\epsilon \to 0}\inf_{F\ni \mu}\limsup_{n \to \infty}\frac{1}{n}\log s(X_{n,F},d_n,\epsilon).
		\end{align}
		Fix $\epsilon >0$ and $0<\alpha \leq \frac{1}{2}$, and let $F$ be a neighborhood of $\mu$. Choose $\delta \in (0, \frac{\alpha \epsilon}{2})$, and then take $E$  to be  a $(\overline{d}_n,\delta)$-spanning set of $X_{n,F}$ with the smallest cardinality $N:=r(X_{n,F},\overline{d}_n,\delta)$. Then, by a standard  approach, we can construct a finite Borel partition $\mathcal{P}=\{P_1,..,P_N\}$ of $X_{n,F}$ with $\diam(\mathcal{P},\overline{d}_n)<2\delta$, and each atom of $\mathcal{P}$ is non-empty. Choose  arbitrarily a point $x_j$ from $P_j$.  Let $\mathcal{Q}=\{Q_1,...,Q_m\}$ be a finite Borel partition of $X$  with $\diam(\mathcal{Q},d)<\epsilon$. For each $1\leq j \leq N$ and $x\in P_j$, we  define $$\omega(x,n)=(\omega_0(x),...,\omega_{n-1}(x))\in \{0,1,...,m\}^n,$$ where $\omega_k(x)=0$ if $d(T^kx,T^kx_j)< \frac{\epsilon}{2}$; otherwise, we let $\omega_k(x)=t$ satisfy $T^kx\in Q_t$. If  $x,y\in P_j$ and $\omega(x,n)=\omega(y,n)$, then $d_n(x,y)<\epsilon$.  Hence, for  any $(d_n,\epsilon)$-separated set $T_j$ of $P_j$ with the largest cardinality, one has $\omega(x,n)\not=\omega(y,n)$ for any  distinct $x,y \in T_j$. Noticing that $\overline{d}_n(x,x_j)<\delta$  for any $x\in P_j$,  by the choice of $\delta$ we have
		$$\#\{k:\omega_k(x)\not=0\}\leq n\alpha.$$
		Then  the cardinality of $T_j$ is bounded above by
		$$\#T_j\leq \sum_{j\leq n\alpha}\binom{n}{j}m^j\leq\sum_{j\leq n\alpha}\binom{n}{j}m^{n\alpha}\leq2^{n\gamma(\alpha)}m^{n\alpha},$$
		where $\gamma(\alpha)=-\alpha \log_2\alpha-(1-\alpha)\log_2(1-\alpha)$, and we used  a  combinatorial lemma  \cite[Lemma 2.1]{ps05}  for the third  inequality. Thus, we get
		\begin{align*}
			s(X_{n,F},d_n,\epsilon)\leq  r(X_{n,F},\overline{d}_n,\delta)\cdot e^{n(\gamma(\alpha)\log2 +\alpha \log m)},
		\end{align*}
		which implies that
	\begin{align*}
	&\limsup_{n \to \infty}\frac{1}{n}\log s(X_{n,F},d_n,\epsilon)\\
	\leq& \limsup_{n \to \infty}\frac{1}{n}\log r(X_{n,F},\overline{d}_n,\delta)+ (\gamma(\alpha)\log2 +\alpha \log m).
	\end{align*}
		Taking the infima over all neighborhoods \(F\) of \(\mu\) and letting \(\alpha \to 0\) (hence \(\delta \to 0\)), we have
		\begin{align*}
		&\inf_{F\ni \mu}\limsup_{n \to \infty}\frac{1}{n}\log s(X_{n,F},d_n,\epsilon)\\
		\leq& \lim_{\delta \to 0}\inf_{F\ni \mu}\limsup_{n \to \infty}\frac{1}{n}\log r(X_{n,F},\overline{d}_n,\delta),
		\end{align*}   and hence yields  the desired  inequality (\ref{equ 3.9}) by letting $\epsilon \to 0$.  So the equality  (\ref{equ 3.8}) holds. Therefore, we have
		$$\lim_{\epsilon \to 0}{PS}_\mu(T,\epsilon)=\lim_{\epsilon \to 0}\inf_{F\ni \mu}\limsup_{n \to \infty}\frac{1}{n}\log s(X_{n,F},\overline{d}_n,\epsilon).$$
		
		By \cite[Proposition 4.2]{w21}, for any $\mu \in M(X,T)$\footnote[4]{Although the statement is given for ergodic measures, the proof applies to  all invariant measures.}, one has that, for any $L>2$ and $\epsilon >0$, 
		$$\inf_{F\ni \mu}\limsup_{n \to \infty}\frac{1}{n}\log s(X_{n,F},\overline{d}_n,\epsilon)\leq \frac{L}{L-1}R_{\mu,L^1}(\frac{1}{6L+2}\epsilon).$$
		Letting $\epsilon \to 0$ and then letting $L\to \infty$, we get  $$\lim\limits_{\epsilon \to 0}PS_\mu(T,\epsilon)\leq h_{\mu,L^1}(T).$$

		\emph{Step 2.} For any $\mu \in M(X,T)$, we show that 
		\begin{itemize}
			\item [(1)] $h_{\mu,L^1}(T)\leq h_{\mu,B}(T)$;
			\item [(2)]  $h_{\mu,L^P}(T)\leq h_{\mu,L^{\infty}}(T)$  for all $1\leq p <\infty$.
		\end{itemize}
		
		$(1)$. It follows from  the fact $R_{\mu,L^1}(\epsilon)\leq R_{\mu,B}(\epsilon)$ for  $\epsilon >0$.
		
		$(2)$.
		Now assume that  the random variables $\xi$ and $\eta$ satisfy the  distortion condition:
		$$\mathbb{E}\left( \text{the number of } 0\leq k\leq n-1~\text{with}~d(T^k\xi,\eta_k)\geq \epsilon\right)\leq nr.$$ 
		For sufficiently small $r>0$, we have
		$$\mathbb{E}\left(\frac{1}{n}\sum_{k=0}^{n-1}d(\xi_k,\eta_k)^p\right)\leq \epsilon^p+r\cdot\diam(X,d)^p<(2\epsilon)^p.$$
		This implies that $ R_{\mu,L^{p}}(2\epsilon)\leq \lim_{r \to 0} R_{\mu,L^{\infty}}(\epsilon,r)$, and hence $h_{\mu,L^P}(T)\leq h_{\mu,L^{\infty}}(T)$.

		\emph{Step 3.} For any $\mu \in M(X,T)$, we show that
		$$h_{\mu,B}(T)\leq h_{\mu}(T)$$
		
		Similar to the  proof of (\ref{inequ 3.6}),  for every $\epsilon >0$ we have $$R_{\mu,B}(2\epsilon) \leq \inf_{\diam  (\alpha) \leq \epsilon, \alpha \in \mathcal{P}_{X}}\limits h_\mu(T,\alpha).$$
		Since $$\lim_{\epsilon  \to 0} \inf_{\diam  (\alpha) \leq \epsilon, \alpha \in \mathcal{P}_{X}} h_\mu(T,\alpha)=h_{\mu}(T)$$ for all $\mu \in M(X,T)$, this implies that $h_{\mu,B}(T)\leq h_{\mu}(T)$.
		
		We complete the proof by  Lemma  \ref{lem 3.3} and these inequalities  stated in steps 1-3. 
	\end{proof}

	\begin{ex}
		\begin{itemize}
			\item [(1)] By Brin-Katok's entropy formula {\rm \cite{bk83}}, every translation of a compact metrizbale group $G$ with the Haar measure $\mu$ has zero  measure-theoretic entropy. Hence, we  have  $$h_{\mu,L^p}(T)=h_{\mu,L^\infty}(T)=h_{\mu,B}(T)=0.$$
			\item [(2)]  Let $\mu$ be the product measure of the $(p_0,...,p_{k-1})$-shift over $\{0,...,k-1\}^{\mathbb{Z}}$. Then $$h_{\mu,L^p}(T)=h_{\mu,L^\infty}(T)=h_{\mu,B}(T)=\sum_{j=0}^{k-1}-p_j \log p_j.$$
		\end{itemize}
	\end{ex}
	
	We remark that, in the Step 1 of the proof of Theorem \ref{thm 1.2}, the \(g\)-almost product property  is  used to  obtain a lower bound  for the \(L^p\) rate-distortion entropy of  non-ergodic invariant  measures.   It remains unclear whether there exists  a TDS without the $g$-almost product property such that for some   non-ergodic  invariant measure,  its  rate-distortion entropies are strictly less than the measure-theoretic entropy. The following example shows that the g-almost product property is merely a sufficient condition for the validity of Theorem 1.2, at least on a specific measure.
	
	\begin{ex}

		Let $\Sigma_2=\{0,1\}^{\Z}$ be a product space equipped with a compatible metric
		\[
		d_0(x,y)=\sum_{n\in\Z}\frac{|x_n-y_n|}{2^{|n|}},
		\]
		and let $(\Sigma_2,d_0,\sigma)$ be the usual symbolic system. Consider two disjoint copies of $\Sigma_2$: $X_1=\Sigma_2\times\{1\}$ and $X_2=\Sigma_2\times\{2\}$.  Now endow $X=X_1\cup X_2$ with the metric
		$$d\big((x,i),(y,j)\big)=d_0(x,y)+|i-j|,$$
		and define  a homeomorphism map $T:X\to X$   by
		\[
		T(x,i)=(\sigma x,i),\quad i=1,2,
		\]
		where $\sigma: \Sigma_2 \rightarrow \Sigma_2$ is the left shift.  
		Then $(X,d,T)$ is a TDS.
		
		Let $P=(\frac{1}{2},\frac{1}{2})^{\mathbb{Z}}$ be  the Bernoulli probability measure on  $\Sigma_2$. For $i=1,2$, we  put
		$\nu_i(E\times \{i\})=P(E)$ for  every Borel set $E$ of $\Sigma_2$, and define $\mu_i(A)=\nu_i(A\cap X_i)$ for any Borel set of $X$. Both $\mu_1$ and $ \mu_2$ are ergodic w.r.t. $T$ and satisfy $h_{\mu_1}(T)=h_{\mu_2}(T)=\log 2$. We let 
		\[
		\mu=\frac12\mu_1+\frac12\mu_2.
		\]
		Then $\mu$ is  a non-ergodic $T$-invariant measure and $h_\mu(T)=\frac12 h_{\mu_1}(T)+\frac12 h_{\mu_2}(T)=\log 2.$
		
		We claim that  $(X,d,T)$ has no $g$-almost product property and 
		$$h_{\mu,L^p}(T)=h_{\mu,L^\infty}(T)=h_{\mu,B}(T)=\log2.$$
		
		i) Assume that there exists a mistake function $g$ and a non-increasing function $m:\R^+\to\N$ such that the $g$-almost product property holds for $(X,d,T)$. Take $\epsilon=\frac{1}{4}$,   choose any a point $x_1\in X_1$, $x_2\in X_2$, and set $n_1=m(\epsilon)$, $n_2=m(\epsilon)$.  Then the $g$-almost product  property implies that 
		\[
		S:=B_{n_1}(g;x_1,\epsilon)\cap T^{-n_1}B_{n_2}(g;x_2,\epsilon)\not=\emptyset.
		\]
		Pick up $y\in S$. This means that in the time interval $[0,n_1-1]$, we have $d(T^jy,T^jx_1)<\epsilon$ for  at least $n_1-g(n_1)$ many $j$. Noticing that $T^jx_1\in X_1$ and ${\rm dist}(X_1,X_2)\geq1$, we must have $y\in X_1$. Besides, the fact  $T^{n_1}y\in B_{n_2}(g;x_2,\epsilon)$ forces  that
		$$1\leq  d(T^{n_1+j}y,T^jx_2)<\epsilon$$
		for  at least $n_2-g(n_2)$ indices $j\in [n_1,n_1+n_2-1]$, a contradiction with $\epsilon <1$. Hence, the $g$-almost product property fails for $(X,d,T)$.

		ii)  By the proof of Theorem \ref{thm 1.2}, $$h_{\mu,L^p}(T),h_{\mu,L^\infty}(T),h_{\mu,B}(T)$$ are all bounded above by $h_{\mu}(T)$.Therefore, we  only need to show  the rate-distortion entropies of $\mu$ is no less than $\log 2$.	We only treat the $L^p$ case, and the similar arguments can be applied to  the  rate-distortion entropies $h_{\mu,L^\infty}(T)$ and $h_{\mu,B}(T)$.
		
		Now  let $\xi$ and $\eta=(\eta_0,\dots,\eta_{n-1})$ be random variables on a probability space $(\Omega, \mathbb{P})$  such that $\text{\rm{Law}}(\xi)=\mu$ and  $\eta$  takes values in $X^n$ satisfying
		\[
		\E\Big(\frac{1}{n}\sum_{k=0}^{n-1}d(T^k\xi,\eta_k)^p\Big)<\epsilon^p.
		\]
		
		For $i=1,2$, put $A_i=\{\xi\in X_i\}$ and let $ \mathbb{P}_{{A_i}}(\cdot)=\frac{\mathbb{P} (\cdot~ \cap  A_i)}{\mathbb{P}(A_i)}$, where $\mathbb{P}(A_i)=1/2$.  For the probability space  $(A_i, \mathbb{P}_{A_i})$,  noticing that the  random variable $\xi$ restricted on $A_i$ obeys the  law $\mu_i$, we write  
		\begin{align*}
			\E\Big(\frac{1}{n} \sum_{k=0}^{n-1}d(T^k\xi,\eta_k)^p\mid A_i\Big):&=
			\int_{A_i} \sum_{k=0}^{n-1}d(T^k\xi(\omega),\eta_k(\omega))^p  d \mathbb{P}_{A_i}(\omega)\\
			&=\frac{1}{\mathbb{P}(A_i)}\int_{\Omega} \sum_{k=0}^{n-1}d(T^k\xi(\omega),\eta_k(\omega))^p  d \mathbb{P}(\omega)
		\end{align*}
		to denote the expectation, and  write  $I(\xi;\eta\mid A_i)$ for the mutual information. Then we have
		$$\E\Big(\frac{1}{n} \sum_{k=0}^{n-1}d(T^k\xi,\eta_k)^p\mid A_i\Big)<2\epsilon^p\leq (2\epsilon)^p,$$
		and hence
		\[
		\frac1n I(\xi;\eta\mid A_i)\ge R_{\mu_i,L^p}(2\epsilon).
		\]
		
		Consider  a random variable $Z:=f(\xi)$ on $(\Omega,\mathbb{P})$  taking  values $\{0,1\}$, where $f(x)=1$ if $x\in X_1$ and $f(x)=0$ if $x\in X_2$.  Since $Z$ is a function of $\xi$, the chain rule for mutual information gives
		\begin{align*}
			I(\xi;\eta)=I(\xi,Z;\eta)&=I(Z;\eta)+I(\xi;\eta\mid Z)\\
			&\geq  I(\xi;\eta\mid Z)\\
			&=\mathbb{P}(A_1)I(\xi;\eta\mid A_1)+\mathbb{P}(A_2)I(\xi;\eta\mid A_2)\\
			&=\frac12 I(\xi;\eta\mid A_1)+\frac12 I(\xi;\eta\mid A_2).
		\end{align*}
		Consequently, we have
		$$\frac{1}{n}	I(\xi;\eta)\geq \frac{{1}}{2}(R_{\mu_1,L^p}(2\epsilon)+R_{\mu_2,L^p}(2\epsilon)),$$
		which implies that $R_{\mu,L^p}(\epsilon)\geq \frac{{1}}{2}(R_{\mu_1,L^p}(2\epsilon)+R_{\mu_2,L^p}(2\epsilon))$.	
		By Theorem  \ref{thm 1.2} (1), we have  $\lim_{\epsilon\to0}R_{\mu_i,L^p}(\epsilon)=h_{\mu_i}(T)=\log2$. This yields that 
		\[
		\lim_{\epsilon\to0}R_{\mu,L^p}(\epsilon)\ge\log2.
		\]
		To sum up,  we get $	h_{\mu,L^p}(T)=h_{\mu}(T)=\log2.$	
	\end{ex}

	\subsection{Ergodic Lindenstrauss-Tsukamoto's double variational principle}
	
	In this subsection, we  prove Theorem \ref{thm 1.3}.

	A TDS  $(X, T)$ is said to have the \emph{marker property}  if for every  $n \geq 1$, there exists an open set $U\subset X$  such that
	$$U\cap T^j U=\emptyset, 1\leq j\leq n,  \mbox{and}~X=\cup_{n\in \mathbb{Z}} T^n U.$$
	
	The examples include  aperiodic minimal systems \cite[Lemma 3.3]{l99},  aperiodic finite-dimensional  systems \cite[Theorem 6.1]{g15},  and  the extension of an aperiodic system  which has a countable number of minimal subsystems \cite[Theorem 3.5]{g17}.   It is easy to  see that the  marker property implies  aperiodicity. The converse is false since there exist aperiodic systems without the marker property \cite{tty22,s23}. 
	The  marker property finds applications in  embedding problems of dynamical systems (cf. \cite{lw00,g15, glt16, g17,gt20}). For instance,   it  was  shown that \cite[Theorem 3.12]{lt19} for  any TDS $(X,T)$  with the marker property, there exists a metric $d \in \mathcal{D}(X)$ such that 
	\begin{align}\label{equ 2.4}
		\mdim(T,X)= \overline{\rm mdim}_M(T,X,d).
	\end{align}
	
	Besides, for systems with the marker property Lindenstrauss and Tsukamoto \cite{lt19} proved  a double variational principle for mean dimension:
	\begin{align*}
		\mdim(T,X)=\min_{d\in \mathcal{D}(X)} \sup_{\mu \in M(X,T)}\overline{\rm {rdim}}_{L^1}(X,T,d,\mu).
	\end{align*}
	Additionally, if the system  has finite mean dimension, we show that the  supremum  suffices to take  over the set of ergodic measures and that the result is  valid for other types of measure-theoretic  $\epsilon$-entropy  considered in $\mathcal{E}$.   This is realized by  introducing the following  auxiliary quantity:
	
	\begin{df}
		Let $(X,T)$ be a TDS with a metric $d\in \mathcal{D}(X)$. For every $\mu \in M(X,T)$, we define  
		$$F(\mu,d):=\limsup_{\epsilon \to 0}\frac{1}{\logf}\sup_{\diam (\UU)\leq \epsilon,  \Leb (\UU)\geq \frac{\epsilon}{4}}h_{\mu}(T,\UU),$$
		where the supremum ranges over all finite open covers of $X$ with $\diam (\UU)\leq \epsilon$ and $ \Leb (\UU)\geq \frac{\epsilon}{4}$.
	\end{df}
	
	\begin{lem}\label{lem 3.7}
		Let $(X,T)$ be a TDS with a metric $d\in \mathcal{D}(X)$. Then 
		
		(1)  for  every $\mu \in E(X,T)$ and   $\epsilon >0$, we have
		\begin{align*}
		\overline{h}_\mu^{K}(T,2\epsilon) \leq & \sup_{\diam (\UU)\leq \epsilon,  \Leb (\UU)\geq \frac{\epsilon}{4}}h_{\mu}(T,\UU)\\
		=&\sup_{\diam (\UU)\leq \epsilon,  \Leb (\UU)
			\geq  \frac{\epsilon}{4}}h_{\mu}^S(\UU) \leq  \overline{h}_{\mu}^K(T, \frac{\epsilon}{4}).
		\end{align*}
		
		Consequently, for every $\mu \in E(X,T)$, $$F(\mu,d)=\limsup_{\epsilon \to 0}\frac{h_{\mu}(T,\epsilon)}{\logf}$$ holds  for all  $h_{\mu}(T,\epsilon) \in \mathcal{E}$. Furthermore, the variational principle $$\overline{\rm mdim}_M(T,X,d)=\limsup_{\epsilon \to 0} \frac{1}{\logf}\sup_{\mu \in E(X,T)} \sup_{\diam (\UU) \leq \epsilon, \atop \Leb(\UU)\geq \frac{\epsilon}{4}}h_\mu(T,\mathcal{U})$$
		is  valid for $\overline{\rm mdim}_M(T,X,d)$.
		
		(2) If   $\overline{\rm mdim}_M(T,X,d)<\infty$, then 
		\begin{align*}
			\sup_{\mu \in M(X,T)}F(\mu,d)=\sup_{\mu \in E(X,T)}F(\mu,d).
		\end{align*}
	\end{lem}
	\begin{proof}
		(1).
		By  Lemma \ref{lem 2.1} (1), we have $h_{\mu}(T, \UU)= h_\mu^S(\mathcal{U})$. Take a   finite open cover $\UU$   of $X$  with  $\diam (\UU)\leq \epsilon(<2\epsilon)$ and  $\Leb(\UU)\geq \frac{\epsilon}{4}$. By comparing the definitions,  we have 
		$$\overline{h}_\mu^{K}(T,2\epsilon) \leq \sup_{\diam (\UU) \leq \epsilon, \Leb(\UU)\geq \frac{\epsilon}{4}}h_\mu^S(\mathcal{U})
		\leq \overline{h}_\mu^{K}(T,\frac{\epsilon}{4}).$$
		
		The remaining statements hold by Lemma \ref{lem 2.1} (3).
		
		(2). Assume that  $\overline{\rm mdim}_M(T,X,d)<\infty$. It suffices to  verify that $$\sup_{\mu \in M(X,T)}F(\mu,d)
		\leq \sup_{\mu \in E(X,T)}F(\mu,d).$$
		
		Now fix $\mu \in M(X,T)$ and let $\mu=\int_{E(X,T)}md\tau (m)$ be the ergodic decomposition of $\mu$, where $\tau$ is a (unique) Borel probability measure on $E(X,T)$. We
		choose a sequence  $\{\epsilon_k\}$ of positive real numbers satisfying $\epsilon_k \to 0$ as $k\to \infty$ and a family $\{\UU_k\}$ of  finite open covers of $X$  with  $\diam (\UU_k)\leq \epsilon_k$ and $\Leb(\UU_k)\geq \frac{\epsilon_k}{4}$ such that
		$$F(\mu,d)=\lim_{k \to \infty}\frac{1}{\log \frac{1}{\epsilon_k}}h_{\mu}(T,\UU_k).$$ Without loss of generality, by (1) we may require  that for the sequences $\{\epsilon_k\}$ and $\{\UU_k\}$, it holds that for all $m \in E(X.T)$,
		$$\frac{1}{\log \frac{1}{\epsilon_k}}h_{m}(T,\UU_k)<\overline{\rm mdim}_M(T,X,d)+1<\infty.$$ Let $\gamma >0$ and $A_{\gamma}:=\min\{F(\mu,d)-\gamma, \frac{1}{\gamma}\}$. Then
		\begin{align*}
			A_{\gamma}&<\lim_{k \to \infty}\frac{1}{\log \frac{1}{\epsilon_k}}\int_{E(X,T)}h_{m}(T,\UU_k)
			d\tau(m) ~~ \text{by  Lemma \ref{lem 2.1} (2)}\\
			&=\limsup_{k \to \infty}\int_{E(X,T)}\frac{1}{\log \frac{1}{\epsilon_k}}h_{m}(T,\UU_k)
			d\tau(m)\\
			&\leq \int_{E(X,T)}\limsup_{k \to \infty} \frac{1}{\log \frac{1}{\epsilon_k}}h_{m}(T,\UU_k)
			d\tau(m),
		\end{align*}
		where we used a variant of the classical Fatou's lemma\footnote[5]{We present it  for the sake of readers: let $\{f_n\}$ be a sequence of measurable functions on a probability  space $(\Omega, \mathcal{F},\mathbb{P})$. If for some integrable function  $g: \Omega \rightarrow \mathbb{R}$,  for every $n\geq 1$, $f_n\leq g$ for $\mathbb{P}$-a.e. $\omega\in \Omega$, then
			$$\limsup_{n \to \infty}\int_{\Omega} f_n d \mathbb{P}\leq  \int_{\Omega} \limsup_{n \to \infty} f_nd\mathbb{P}.$$} for the  inequality.   
		Hence, there exists  $m\in E(X,T)$ such that $A_{\gamma}<F(m,d) \leq  \sup_{\mu \in E(X,T)}F(\mu,d).$
		Letting $\gamma \to 0$, since $\mu$ is arbitrary,  we get the desired inequality. 
	\end{proof}

	Motivated by the proof of Lemma \ref{lem 3.7}, we take this opportunity to establish an analogous ergodic decomposition theorem for the mean R\'enyi information dimensions and their variants defined by  the following candidate set of measure-theoretic $\epsilon$-entropy: 
	\begin{align*}
		\mathcal{E}_1= \left\{
		{
			\displaystyle
			{
				\inf_{\diam (\alpha) < \epsilon, \atop \alpha \in \mathcal{P}_{X}}\limits  h_{\mu}(T, \alpha),\quad
				\inf_{\diam (\alpha) \leq \epsilon, \atop \alpha \in \mathcal{P}_{X}}\limits h_{\mu}(T, \alpha)
				\atop
				\inf_{\diam (\alpha)\leq \epsilon,\ \mu(\partial \alpha)=0 \atop \alpha \in \mathcal{P}_{X}}\limits h_{\mu}(T, \alpha),\quad
				\inf_{\diam (\mathcal{U}) \leq \epsilon} \limits h_{\mu}(T, \mathcal{U})
			}
			\atop
			\sup_{\diam (\mathcal{U}) \leq \epsilon,\ \mathrm{Leb}\,(\mathcal{U}) \geq \frac{\epsilon}{4}}\limits h_{\mu}(T, \mathcal{U}),\quad
			\sup_{\mathrm{Leb}\,(\mathcal{U}) \geq \epsilon }\limits  h_{\mu}(T, \mathcal{U})
		}
		\right\}
	\end{align*}
	In addition,  for $\mathbb{R}^d$-actions, Tsukamoto proposed the notion of rate-distortion dimension in \cite{t25a}, wherein a double inequality connecting mean dimension and rate-distortion dimension is  established, and the corresponding ergodic decomposition theorems for this dimension are presented in \cite[Theorems 1.1-1.2]{t25b}.
	
	\begin{thm}
		Let  $(X,d,T)$ be a TDS and   $\mu \in M(X,T)$, and let  $\mu=\int_{E(X,T)} m\tau (m)$ be the ergodic decomposition of $\mu$, where   $\tau$ is a unique Borel probability measure on $E(X,T)$. Then for every $h_{\mu}(T,\epsilon) \in \mathcal{E}_1$,
		\begin{itemize}
			\item [(1)] $\int_{E(X,T)}\liminf_{\epsilon \to 0} \limits \frac{h_{\mu}(T, \epsilon)}{\logf}\tau(m) \leq  \liminf_{\epsilon \to 0} \limits \frac{h_{\mu}(T,\epsilon)}{\logf}$.
			\item [(2)]  If   $\overline{\rm mdim}_M(T,X,d)<\infty$, then $$\int_{E(X,T)}\limsup_{\epsilon\to 0} \limits\frac{h_{m}(T,\epsilon)}{\log\frac{1}{\epsilon}} \tau(m) \geq  \limsup_{\epsilon \to 0} \limits \frac{h_{\mu}(T,\epsilon)}{\logf}.$$
		\end{itemize}
		Consequently, if  $\limsup_{\epsilon \to 0} \limits \frac{h_{\mu}(T, \epsilon)}{\logf}=\liminf_{\epsilon \to 0} \limits \frac{h_{\mu}(T, \epsilon)}{\logf}$ for $\tau$-a.e. $m\in E(X,T)$,  we have $$\lim_{\epsilon \to 0} \limits \frac{h_{\mu}(T,\epsilon)}{\logf}=\int_{E(X,T)}\lim_{\epsilon \to 0} \limits \frac{h_{\mu}(T, \epsilon)}{\logf}\tau(m).$$
		
	\end{thm}
	
	\begin{proof}
		
		By  \cite[Lemma 2.3 (3)]{fw16}, for every $\epsilon>0$, 
		$$\mu \in M(X,T) \mapsto h_{\mu}(T,\epsilon):=\inf_{
			\diam (\alpha)< \epsilon,  \alpha \in \mathcal{P}_{X}} \limits h_{\mu}(T,\alpha)$$
		is  upper semi-continuous, and hence is Borel measurable.  
		Putting $t=\logf$, it holds that  for every $\mu \in M(X,T)$,
		\begin{align*}
			\limsup_{\epsilon \to 0}\frac{h_{\mu}(T,\epsilon)}{\logf}=\limsup_{t \to \infty}\frac{h_{\mu}(T, e^{-t})}{t}=\limsup_{\mathbb{N} \ni n \to \infty}\frac{h_{\mu}(T, e^{-n})}{n},
		\end{align*}
		where we used the fact $h_{\mu}(T,e^{-t_1})\geq h_{\mu}(T,e^{-t_2})$ if $t_1\geq t_2$ for the second equality. Moreover, the equality  is still  valid  by changing the upper limits into the lower limits. Therefore,  we deduce that the maps
		\begin{align*}
			&\mu \in M(X,T)\mapsto \limsup_{\epsilon \to 0}\frac{h_{\mu}(T,\epsilon)}{\logf},\\
			&\mu \in M(X,T)\mapsto \liminf_{\epsilon \to 0}\frac{h_{\mu}(T,\epsilon)}{\logf},
		\end{align*}  
		are Borel measurable. It further verifies the involved  measurability  for the other candidates in $\mathcal{E}_1$ by  two inequalities.  One holds for all $\mu\in M(X,T)$ and $\varepsilon>0$,
		$$h_{\mu}(T,2\epsilon) \leq \inf_{
			\diam  (\alpha) \leq \epsilon, \atop \alpha \in \mathcal{P}_{X}}\limits h_{\mu}(T, \alpha) \leq h_{\mu}(T,\epsilon).$$
		The other follows from \cite[Lemma 3.23 (1)]{y25}, stating that   for every $\mu \in M(X,T)$ and $\epsilon >0$, 
		\begin{align}\label{inequu 3.11}
			\inf_{
				\diam  (\alpha) \leq \epsilon, \atop \alpha \in \mathcal{P}_{X}}\limits h_{\mu}(T, \alpha)&\leq \inf_{
				\diam (\alpha)\leq \epsilon,  \mu(\partial \alpha)=0 \atop \alpha \in \mathcal{P}_{X}} \limits h_{\mu}(T,\alpha) \leq \inf_{\diam (\UU) \leq \epsilon} \limits  h_{\mu}(T,\UU) \nonumber\\
			&\leq  \sup_{\diam (\UU) \leq \epsilon,  \Leb (\UU) \geq \frac{\epsilon}{4} }\limits h_{\mu}(T,\UU) \leq  \sup_{\Leb (\UU) \geq \frac{\epsilon}{4} } \limits h_{\mu}(T,\UU) \nonumber\\
			&\leq \inf_{
				\diam (\alpha)\leq \frac{\epsilon}{8}}\limits h_{\mu}(T,\alpha).
		\end{align}

		(1). In view of \eqref{inequu 3.11}, it suffices to prove  that the ergodic decomposition theorem holds for $$\inf_{
			\diam  (\alpha) < \epsilon, \atop \alpha \in \mathcal{P}_{X}}\limits h_{\mu}(T, \alpha).$$ Choose a decreasing  sequence  $\epsilon_n \to 0$ as $n \to \infty$ and  a sequence  $\{\alpha_n\}$ of finite Borel partitions of $X$ with $\diam(\alpha_n)\leq \epsilon_n$ such that 
		\begin{align}\label{eq 3.11}
			\lim_{n\to \infty}\frac{1}{\log\frac{1}{\epsilon_n}} h_{\mu}(T, \alpha_n)=\liminf_{\epsilon \to 0}\frac{1}{\logf}\inf_{\diam (\alpha)\leq \epsilon}\limits h_{\mu}(T, \alpha).
		\end{align}	 
		We show that  the entropy map
		$$\mu \in M(X,T)\mapsto h_{\mu}(T,\alpha)$$
		is Borel measurable. 
		
		Recall that $ h_{\mu}(T,\alpha)=\lim_{n \to \infty}\frac{1}{n}H_{\mu}(\alpha^n)$ for every $\mu \in M(X,T)$.  It follows from the definition of measure-theoretic entropy that we only need to show
		for every $A \in \mathcal{B}(X)$,
		$$f_A:\mu \in M(X)\mapsto \mu(A)$$
		is Borel  measurable. We put
		$$\mathcal{F}:=\{A\in \mathcal{B}(X): f_A~\text{is  measurable w.r.t.} ~  \mathcal{B}(M(X))\}.$$
		
		Clearly,  the weak$^{*}$-topology on $M(X)$ implies that  $\mathcal{F}$ contains the  family $\mathcal{S}$  of all open subsets of $X$. Then $\mathcal{S}$  is a $\pi$-system in the sense that  the intersection of any two open sets of $X$ still lies in $\mathcal{S}$. Furthermore,  $\mathcal{F}$ is a $\lambda$-system in the following sense:
		\begin{itemize}
			\item [(1)] obviously, $X\in \mathcal{F}$;
			\item [(2)] if $A, B\in \mathcal{F}$ and $A\subset B$, then
			$$f_{B\backslash A}(\mu)=\mu(B\backslash A)=f_B(\mu)-f_A(\mu),$$
			which implies that $B\backslash A\in \mathcal{F}$;
			\item [(3)] if $A_n\in \mathcal{F}$ and $A_n\subset A_{n+1}$ for every $n\geq 1$, then
			$$f_{\cup_{n\geq 1}A_n}(\mu)=\lim_{n\to \infty}\mu(A_n)=\lim_{n\to \infty}f_{A_n}(\mu).$$
			This implies that $\cup_{n\geq 1}A_n \in \mathcal{F}$.  
		\end{itemize} 
		The Dynkin's $\pi-\lambda$ Theorem\footnote[6]{It says that for any $\pi$-system $\mathcal{C}$ and $\lambda$-system $\mathcal{D}$ in a measurable space $S$, if $\mathcal{C}\subset \mathcal{D}$,  we have $\sigma(\mathcal{C})\subset \mathcal{D}$.} (cf.  \cite[Theorem 1.1]{kal21}) implies that
		$$\mathcal{B}(X)=\sigma(\mathcal{S})\subset \mathcal{F}\subset  \mathcal{B}(X).$$
		
		The Jacobs's theorem  (cf. \cite[Theorem 8.4]{w82}) yields that for every $n\geq 1$, $$h_{\mu}(T,\alpha_n)= \int_{E(X,T)} h_{m}(T,\alpha_n) d\tau(m).$$
		Using (\ref{eq 3.11}) and the  Fatou's lemma, we obtain  that
		\begin{align*}
			&\liminf_{\epsilon \to 0}\frac{1}{\logf}\inf_{\diam (\alpha)\leq \epsilon}\limits h_{\mu}(T, \alpha)\\ 
			=&  \liminf_{n \to \infty}\int_{E(X,T)} \frac{1}{\log\frac{1}{\epsilon_n}} h_{m}(T, \alpha_n) d\tau(m),\\
			\geq& \int_{E(X,T)} \liminf_{n \to \infty}\frac{1}{\log\frac{1}{\epsilon_n}} h_{m}(T, \alpha_n) d\tau(m),\\
			\geq& \int_{E(X,T)} \liminf_{\epsilon \to 0}\frac{1}{\logf}\inf_{\diam (\alpha)\leq \epsilon}\limits h_{m}(T, \alpha) d\tau(m),
		\end{align*}
		which implies  the desired inequality.

		(2).  It follows from the proof of Lemma \ref{lem 3.7} (2). 
	\end{proof}
	
	Using Lemma \ref{lem 3.7}, we prove Theorem \ref{thm 1.3}.
	
	\begin{proof}[Proof of Theorem \ref{thm 1.3}]
		We divide the proof into two steps:
		
		\emph{Step 1.} For  every  $d\in \mathcal{D}^{'}(X)$ and  $h_{\mu}(T,\epsilon)\in \mathcal{E}\cup \{R_{\mu, L^p}(\epsilon)\}$, we have
		$$\mdim(T,X)\leq  \sup_{\mu \in E(X,T)}\{ \limsup_{\epsilon \to 0}\frac{1}{\logf}h_{\mu}(T,\epsilon)\}.$$
		
		Fix  a metric $d\in \mathcal{D}^{'}(X)$. By \cite[Lemma 3.10]{lt19},  there exists  a   metric $d^{'} (\leq d)\in  \mathcal{D}(X)$   admitting  the tame growth of covering numbers. The known results in \cite[Proposition 3.2, Theorem 3.11]{lt19} imply that
		\begin{align}\label{ineq 3.11}
			\mdim(T,X)\leq  \sup_{\mu \in M(X,T)}\overline{\rm {rdim}}_{L^1}(X,T,d^{'},\mu).
		\end{align}
		Recall that in (\ref{equ 2.1}) and (\ref{equ 2.2}),   it holds that for every   $\mu \in M(X,T)$ and $\epsilon >0$, 
		\begin{align*}
			R_{\mu, L^1}(2\epsilon)\leq &\inf_{\diam  (\alpha, d^{'}) \leq \epsilon, \atop \alpha \in \mathcal{P}_{X}}h_\mu(T,\alpha)\\
			 \leq&  \sup_{\diam (\UU, d^{'})\leq \epsilon,  \Leb (\UU, d^{'})\geq \frac{\epsilon}{4}}h_{\mu}(T,\UU).
		\end{align*}
		Together with  (\ref{ineq 3.11}) and Lemma \ref{lem 3.7} (2), we obtain 	
		\begin{align}\label{ineq 3.13}
			\mdim(T,X)\leq & \sup_{\mu \in E(X,T)}F(\mu,d^{{'}}) \nonumber\\
			=&\sup_{\mu \in E(X,T)}\overline{\rm {rdim}}_{L^\infty}(X,T,d^{'},\mu),
		\end{align}
		where the equality holds by  Lemma \ref{lem 2.1} (3). The  tame growth of covering numbers of $d^{'}$   \cite[Theorem 1.7]{w21}  ensures  that for every $\mu \in  E(X,T)$,
	\begin{align*}
	\overline{\rm {rdim}}_{L^\infty}(X,T,d^{'},\mu)=&\overline{\rm {rdim}}_{L^1}(X,T,d^{'},\mu) \\
	\leq& \overline{\rm {rdim}}_{L^1}(X,T,d,\mu)\\
	 \leq& \overline{\rm {rdim}}_{L^\infty}(X,T,d,\mu).
	\end{align*}
		Using  Lemma  \ref{lem 2.1} (3) again and the inequality (\ref{ineq 3.13}), we complete  the step 1.

		\emph{Step 2.}
		By Step 1, for every $h_{\mu}(T,\epsilon)\in \mathcal{E}\cup \{R_{\mu, L^p}(\epsilon)\}$  we have
		\begin{align}\label{inequ 3.10}
			\mdim(T,X)&\leq \inf_{d\in \mathcal{D}^{'}(X)} \sup_{\mu \in E(X,T)}\{ \limsup_{\epsilon \to 0}\frac{1}{\logf}h_{\mu}(T,\epsilon)\}\\
			&\leq \inf_{d\in \mathcal{D}^{'}(X)} \{ \limsup_{\epsilon \to 0}\frac{1}{\logf}\sup_{\mu \in E(X,T)} h_{\mu}(T,\epsilon)\}\nonumber\\
			&\leq  \inf_{d\in \mathcal{D}^{'}(X)} \over~~\text{by Lemma \ref{lem 2.1} (3)} \nonumber \\ 
			&=  \min_{d\in \mathcal{D}^{'}(X)} \over \nonumber=\mdim(T,X)~\text{by (\ref{equ 2.4})} \nonumber,
		\end{align}
		where   we used ``$\leq$" to include  $R_{\mu, L^p}(\epsilon)$ for the third inequality. Furthermore,  $\sup_{\mu\in E(X,T)}$ can be replaced by  $\sup_{\mu\in M(X,T)}$ for  (\ref{inequ 3.10}).  
		
		This completes the proof by  Steps 1 and 2.
	\end{proof}
	
	\subsection{Further discussion on variational principles of metric mean dimension}
	
	In this subsection,  using Theorem \ref{thm 1.3} we  briefly discuss the variational principles of metric mean dimension.
	
	\subsubsection{Variational principle for lower Brin-Katok's $\epsilon$-entropy}
	
	In \cite[Problem 1]{shi}, Shi asked whether    $\underline{h}_{\mu}^{BK}(T, \epsilon)$ can be included in  $\mathcal{E}$ presented in  Lemma \ref{lem 2.1} (3). In \cite[Theorem 1.3]{ycz25}, the authors verified that for any TDS $(X,d,T)$, the variational principle  for metric mean dimension is valid for the lower Brin-Katok's $\epsilon$-entropy in terms of  Borel probability measures:
	\begin{align}\label{equ 3.14}
		\over=\limsup_{\epsilon \to 0}\frac{1}{\logf}\sup_{\mu \in M(X)}\underline{h}_{\mu}^{BK}(T,\epsilon).
	\end{align}
	
	We show that for certain dynamical systems, the supremum in \eqref{equ 3.14} can range over the set of invariant measures. This is achieved by  a geometric Frostman's lemma, originally proved for compact subsets of $\mathbb{R}^n$
	and extended to any compact metric space in \cite[Corollary 4.4]{lt19}. which is stated as follows:
	
	\begin{lem}\label{lem 3.2}
		Let $(X,d)$ be a compact metric space. For any $0<c<1$, there exists $\epsilon_0 \in (0,1)$ such that   for any $0< \epsilon \leq \epsilon_0$, there exists $\mu \in M(X)$ satisfying 
		$$\mu(E)\leq \diam(E)^{c\cdot {\rm dim}_H(X,d,\epsilon)}~~~ \forall ~E \subset X~\text{with}~\diam(E)<\frac{\epsilon}{6}.$$
	\end{lem}
	
	Let  ${\rm \overline{dim}}_B(X,d)$ denote the upper box dimension of $X$. Recall that  the Hausdorff dimension of $X$ is defined by ${\rm dim}_H(X,d)=\lim_{\epsilon \to 0}{\rm dim}_H(X,d,\epsilon)$, where
	\begin{align*} 
		{\rm dim}_H(X,d,\epsilon):&=\inf \{s>0: \mathcal{H}_\epsilon^s(X,d)=0\}\\
		&=\sup \{s>0: \mathcal{H}_\epsilon^s(X,d)=\infty\},		
	\end{align*}
	and $$\mathcal{H}_\epsilon^s(X,d):=\inf\{\sum_{n=1}^{\infty} r_n^s: X=\bigcup_{n=1}^{\infty} B_d(x_n,r_n) ~\atop 
	\qquad 	\qquad	\qquad \text{with}~x_n\in X, r_n <\epsilon ~\forall n \geq 1\}.$$ 
	Pontrjagin and  Schnirelmann\footnote[7]{See \cite[Theorem 5.1]{lt19} for an available proof.}  \cite{ps32} proved that for any  compact metrizable space $X$, there exists a metric $d\in \mathcal{D}(X)$ satisfying ${\rm dim}_H(X,d)={\rm \overline{dim}}_B(X,d)$.
	\begin{thm}\label{thm 3.6}
		Let $X$ be a compact metrizable  space and $d$ be a compatible metric such that ${\rm dim}_H(X,d)={\rm \overline{dim}}_B(X,d)$. Then
		\begin{align*}
			{\rm \overline{mdim}}_M(\sigma,X^{\mathbb{Z}},d^{\mathbb{Z}})&=\limsup_{\epsilon \to 0}\frac{1}{\logf}\sup_{\mu \in E(X^{\mathbb{Z}},\sigma)}\underline{h}_{\mu}^{BK}(T,\epsilon)\\
			&= \limsup_{\epsilon \to 0}\frac{1}{\logf}\sup_{\mu \in M(X^{\mathbb{Z}},\sigma)}\underline{h}_{\mu}^{BK}(T,\epsilon),
		\end{align*}
		where $d^{\mathbb{Z}}(x,y)=\sum_{n\in \mathbb{Z}}\frac{d(x_n,y_n)}{2^{|n|}}.$
		
		The results are also valid for ${\rm \underline{mdim}}_M(\sigma,X^{\mathbb{Z}},d^{\mathbb{Z}})$ by changing $\limsup_{\epsilon \to 0}$ into $\liminf_{\epsilon \to 0}$.
	\end{thm}
	
	\begin{proof}
		By \cite[Theorem 5]{vv17}, we have ${\rm \overline{mdim}}_M(\sigma,X^{\mathbb{Z}},d^{\mathbb{Z}})={\rm \overline{dim}}_B(X,d)$. Now let $0<c<1$. Then, by Lemma  \ref{lem 3.2} there exists   exists $\epsilon_0 \in (0,1)$ such that   for any $0< \epsilon \leq \epsilon_0$, there exists $\mu \in M(X)$ such that for any $E \subset X$ with $\diam(E)<\frac{\epsilon}{6}$, we have
		$$\mu(E)\leq \diam(E)^{c\cdot {\rm dim}_H(X,d,\epsilon)}.$$
		
		Let $\nu:=\mu^{\otimes{\mathbb{Z}}}$ be the product measure on $X^{\mathbb{Z}}$. Then $\nu \in  E(X^{\mathbb{Z}},\sigma)$.  For any $x\in  X^{\mathbb{Z}}$, $\epsilon \in (0,\epsilon_0)$ and $n\in \mathbb{N}$,  we define
		$$I_n(x,\epsilon):=\{y\in X^{\mathbb{Z}}: y_j\in B_d(x_j,\epsilon)~~ \forall ~0\leq j<n\}.$$
		Then, for every $x\in X^{\mathbb{Z}}$ we have $B_n(x,\epsilon) \subset I_n(x,\epsilon)$, and hence
		$$\nu(B_n(x,\frac{\epsilon}{14}))\leq \Pi_{0\leq j<n}\mu(B_d(x_j,\frac{\epsilon}{14}))\leq (\frac{\epsilon}{7})^{nc \cdot {\rm dim}_H(X,d,\epsilon)}.$$
		This yields that
		$$ c \cdot {\rm dim}_H(X,d,\epsilon) \cdot \log \frac{7}{\epsilon}\leq \underline{h}_{\nu}^{BK}(\sigma, \frac{\epsilon}{14})\leq  \sup_{\mu \in  E(X^{\mathbb{Z}},\sigma)}\underline{h}_{\mu}^{BK}(\sigma, \frac{\epsilon}{14}).$$
		Taking the upper limits on both sides of the above inequality and letting $c \to 1$, we obtain
		\begin{align*}
			{\rm dim}_H(X,d)&\leq \limsup_{\epsilon \to 0}\frac{1}{\logf}\sup_{\mu \in  E(X^{\mathbb{Z}},\sigma)}\underline{h}_{\nu}^{BK}(\sigma, \epsilon)\\
			&\leq \limsup_{\epsilon \to 0}\frac{1}{\logf}\sup_{\mu \in  E(X^{\mathbb{Z}},\sigma)}\overline{h}_{\mu}^{BK}(\sigma, \epsilon)\\
			&={\rm \overline{mdim}}_M(\sigma,X^{\mathbb{Z}},d^{\mathbb{Z}})={\rm \overline{dim}}_B(X,d).
		\end{align*}
		This completes the proof.
	\end{proof}
	
	We remark that it is still  an  open question \cite{shi,ycz25} whether Theorem  \ref{thm 3.6}  holds for any TDS  in terms of ergodic measures. 
	
	\subsubsection{The  unification problem of variational principles for metric mean dimensions}
	For any  TDS $(X,d,T)$, it always holds that for  any $h_{\mu}(T,\epsilon)\in \mathcal{E}$,
	\begin{align*}
		\over&=\limsup_{\epsilon \to 0}\frac{1}{\logf}\sup_{\mu \in E(X,T)}h_{\mu}(T,\epsilon)\\
		&=\limsup_{\epsilon \to 0}\frac{1}{\logf}\sup_{\mu \in M(X,T)}h_{\mu}(T,\epsilon).
	\end{align*}
	
	Fix a measure-theoretic $\epsilon$-entropy $h_{\mu}(T,\epsilon)\in \mathcal{E}$. Another fascinating question,  which  has been  mentioned several times in  existing references \cite{vv17,lt18,cpv24,ycz25}, is  exchanging the  order of  $\limsup_{\epsilon \to 0}$ and $\sup_{\mu \in M(X,T)}$(or $\sup_{\mu \in E(X,T)}$) in the above variational principles for metric mean dimension.  Unfortunately, for $L^p,L^{\infty}$    rate-distortion functions $h_{\mu}(T,\epsilon)\in \{R_{\mu,L^{p}}(\epsilon), R_{\mu,L^{\infty}}(\epsilon)\}$  Lindenstrauss and Tsukamoto \cite[Section VIII]{lt18} posed an example to show the strict inequality
	\begin{align}\label{inequ 3.11}
		\sup_{\mu \in M(X,T)}\{\limsup_{\epsilon \to 0}\frac{1}{\logf}h_{\mu}(T,\epsilon)\}<\over
	\end{align}
	is possible.
	Therefore, it follows from this example that
	\begin{itemize}
		\item [(1)]  by Lemma \ref{lem 2.1} (3), the order of $\limsup_{\epsilon \to 0}$ and $\sup_{\mu \in E(X,T)}$ can't be exchanged for  other measure-theoretic $\epsilon$-entropy $h_{\mu}(T,\epsilon)\in \mathcal{E}\backslash\{R_{\mu,L^{\infty}}\}$;
		\item [(2)]  for some infinite entropy systems, no maximal metric mean dimension measure exists\footnote[8]{Given $h_{\mu}(T,\epsilon)\in \mathcal{E}$, we say that $\mu \in M(X,T)$ is  a maximal  metric mean dimension  measure \cite{ycz23} if $\mu$ satisfies $\limsup_{\epsilon \to 0}\frac{1}{\logf}h_{\mu}(T,\epsilon)=\over$.},  and every maximal entropy measure for topological entropy\footnote[9]{An invariant measure $\mu \in M(X,T)$ is  called a maximal entropy measure if $h_{top}(T,X)=h_{\mu}(T)$ (cf. \cite[\S 8.3, p.191]{w82}).} is not maximal for  metric mean dimension;
		\item [(3)] the equality for (\ref{inequ 3.11}) can only be expected for certain dynamical systems (e.g., full shifts over  finite-dimensional cubes, and the conservative homeomorphisms \cite{lr24}.). 
	\end{itemize}
	
	Additionally, (aperiodic) systems with the marker property also  offer the possibility of equality in \eqref{inequ 3.11}.

	\begin{thm}\label{thm 2.4}
		Let $(X,T)$ be a TDS admitting the marker property. If $\mdim(T,X)<\infty$, then there exists a metric $d\in \mathcal{D}^{'}(X)$ such that for every $h_{\mu}(T,\epsilon)\in \mathcal{E}$, 
		\begin{align*}
			\mdim(T,X)&={\rm mdim}_{M}(T,X,d)\\
			&= \sup_{\mu \in E(X,T)}\{ \limsup_{\epsilon \to 0}\frac{1}{\logf} h_{\mu}(T,\epsilon)\}\\
			&= \sup_{\mu \in M(X,T)}\{ \limsup_{\epsilon \to 0}\frac{1}{\logf} h_{\mu}(T,\epsilon)\}.
		\end{align*}
	\end{thm}
	
	Under the conditions of Theorem \ref{thm 2.4}, using Lemma \ref{lem 2.1} (3)  these   ergodic variational principles for metric mean dimension reduce to the following unified form:
	\begin{align}\label{inequ 3.13}
		\over=\sup_{\mu \in E(X,T)}\{\limsup_{\epsilon \to 0}\frac{1}{\logf}\inf_{\diam  (\alpha) \leq \epsilon,\atop \alpha \in \mathcal{P}_{X}}\limits h_\mu(T,\alpha)\}.
	\end{align}

	As an ongoing topic on  linking the  ergodic theory  and topological dynamics of infinite entropy systems,  a  proper measure-theoretic metric mean dimension of invariant measures  are supposed to be defined such that  (\ref{inequ 3.13}) holds for all TDSs.  
	
	In the context of the action of amenable groups, using  the amenable measure-theoretic $\epsilon$-entropies,  we  define \emph{a new measure-theoretic metric mean dimension} that does not depend on the choice of the amenable measure-theoretic $\epsilon$-entropies. This allows us to    realize (\ref{inequ 3.13}) by establishing  the variational principles for the amenable metric mean dimension, without imposing the marker property on dynamical systems. This will be done  in a separate work \cite{y25}.

\section*{Acknowledgment}

This work was  completed  at  the Mathematical Center of Chongqing University, which was supported by  the China Postdoctoral Science Foundation (No. 2024M763856) and  the Postdoctoral Fellowship Program of CPSF  (No. GZC20252040).   I would like to thank Prof.\ Hanfeng Li for carefully reading the manuscript and providing many valuable comments and suggestions, which have significantly improved the early draft.
I am also grateful to the editor Prof.\ Tobias Koch and the anonymous referees for their abundant insightful comments and constructive suggestions that greatly improved the quality of this paper.

\ifCLASSOPTIONcaptionsoff
  \newpage
\fi

\begin{IEEEbiographynophoto}{Rui Yang}   received the Ph.D. degree in mathematics from Nanjing
Normal University, China, in 2024. He is a Specially Appointed Associate Professor with the School of Mathematics, Northwest University (Xi'an). His research interests include topological dynamics, ergodic theory, and the applications of the theory of entropy and mean dimension in  information theory and  control systems.
\end{IEEEbiographynophoto}

\end{document}